\documentclass[11pt,reqno]{amsart}

\usepackage[T1]{fontenc}
\usepackage[utf8]{inputenc}
\usepackage{lmodern}
\usepackage{microtype}
\usepackage{amsmath,amssymb,amsthm,mathtools,mathrsfs}
\usepackage{enumitem}
\usepackage{booktabs}
\usepackage{array}
\usepackage{needspace}
\usepackage{xcolor}
\usepackage{float}
\usepackage{tikz}
\usetikzlibrary{arrows.meta,positioning,calc}
\usepackage[colorlinks=true,linkcolor=blue!55!black,citecolor=blue!55!black,
            urlcolor=blue!55!black]{hyperref}
\usepackage[nameinlink,capitalise,noabbrev]{cleveref}
\hypersetup{
  pdftitle={Zeta Renormalization and Pressure at Infinity for an Infinitely Cusped Tree Lattice},
  pdfauthor={Sanghoon Kwon},
  pdfsubject={Local determinants, pressure at infinity, and renormalized zeta functions},
  pdfkeywords={countable Markov shift, pressure at infinity, graph of groups, infinitely cusped tree lattice, local determinant, renormalized Euler product}
}

\numberwithin{equation}{section}

\newtheorem{theorem}{Theorem}[section]
\newtheorem{proposition}[theorem]{Proposition}
\newtheorem{lemma}[theorem]{Lemma}
\newtheorem{corollary}[theorem]{Corollary}

\theoremstyle{definition}
\newtheorem{definition}[theorem]{Definition}

\theoremstyle{remark}
\newtheorem{remark}[theorem]{Remark}

\newcommand{\OrE}{\operatorname{OE}}
\newcommand{\Tr}{\operatorname{Tr}}

\newcommand{\cF}{\mathcal F}
\newcommand{\cC}{\mathcal C}

\newcommand{\one}{\mathbf 1}

\newcommand{\Li}{\operatorname{Li}_2}

\title[Zeta Renormalization and Pressure at Infinity]{Zeta Renormalization and Pressure at Infinity\\
for an Infinitely Cusped Tree Lattice}

\author{Sanghoon Kwon}
\address{Department of Mathematics Education, Catholic Kwandong University,
Republic of Korea}
\email{skwon@cku.ac.kr, shkwon1988@gmail.com}
\date{August 26, 2026}


\begin{document}

\begin{abstract}
We study weighted periodic-orbit zeta functions for an infinitely cusped tree
lattice $Y_q$, where $q\ge2$ is even and the quotient is a one-sided comb.
The global Euler product fails
coefficientwise because infinitely many primitive cycles have length four.
A first-return determinant at a finite directed-edge set nevertheless exists,
and stationary Schur elimination gives an algebraic formula for the root local
zeta and its dominant poles.  For the two-step multiplicity potential we
compute the Gurevich pressure $P_G=2\log(q+1)$ and pressure at infinity
$P_\infty=\log(4q)$, yielding strong positive recurrence and exponential
local-orbit asymptotics.  The height-damped transition operator is trace
class.  After subtraction of an explicit integrated-pressure counterterm,
the inverse Fredholm determinant has a locally uniform finite part, expressed
by a convergent dilogarithmic product and covariant under changes of height.
\end{abstract}

\maketitle

\tableofcontents

\section{Introduction}

For a finite-state dynamical system, periodic orbits are encoded by a zeta
function and, in the Bowen--Lanford setting, by a finite determinant
\cite{BowenLanford1970}.  Bass--Ihara theory supplies the analogous package for
finite graphs and cocompact tree lattices.  Noncompact extensions with a finite
core and finitely many cuspidal rays still retain enough finite-state structure
for a global determinant: examples include Deitmar--Kang's arithmetic tree
quotients \cite{DeitmarKang2018} and Hong--Kwon's geometrically finite graphs
of groups \cite{HongKwon2025}.

We consider a genuinely infinite-cusp regime.  Paulin's description of
geometrically finite tree actions \cite{Paulin2004} and the Markov codings of
Broise-Alamichel and Paulin \cite{BroisePaulinCoding2007,BroisePaulinModular2007}
provide the geometric and symbolic background.  Broise-Alamichel, Parkkonen
and Paulin subsequently developed counting and equidistribution for tree
geodesic flows and constructed lattices with prescribed quotient growth and
prescribed spaces of ends
\cite{BroiseParkkonenPaulin2016,BroiseParkkonenPaulin2019,
BroiseParkkonenPaulin2021}.  One explicit example has a horizontal
ray with a vertical cuspidal ray attached at every vertex.  Its quotient is a
rooted one-sided comb-shaped tree with quadratic vertex-ball growth and
infinitely many cusps.  Its repeated local geometry makes the zeta obstruction
coefficientwise rather than analytic, so there is no global formal series to
continue.
The graph-of-groups realization, finite-covolume calculation and the precise
model-specific meaning of a vertical cusp ray are given in
\cref{prop:lattice-realization,rem:cusp-convention}.

Because the obstruction is coefficientwise, analytic continuation cannot
recover the missing global zeta.  We replace it by a local first-return
determinant and by a renormalized finite part obtained from a height-regularized
family.  The proof combines a stationary Schur complement, pressure at
infinity and a uniform orbit-family Abel estimate; finite-boundary and
translated-orbit formulations separate the general mechanisms from the
geometry of the example.

\paragraph{Analytic and notational conventions.}
All power series in $u$ are centered at $u=0$.  We use the classical
dilogarithm
\begin{equation}\label{eq:dilog-definition}
 \Li(z)=\sum_{k\ge1}\frac{z^k}{k^2}
 =-\int_0^z\frac{\log(1-t)}{t}\,dt,
 \qquad |z|<1.
\end{equation}
Thus $\Li(z)=z+z^2/4+z^3/9+\cdots$.  We use the logarithm branch satisfying
$\log(1)=0$; only its analytic germ at
the origin is needed below.  A series $\sum_j f_j$ of holomorphic functions
converges \emph{normally} on a domain $D$ if
$\sum_j\sup_{u\in K}|f_j(u)|<\infty$ for every compact $K\subset D$.
For products of factors equal to one at the origin, normal convergence means
normal convergence of the corresponding logarithmic series with the branch
vanishing at the origin.
An analytic germ is \emph{algebraic} over a rational-function field if it
satisfies a nonzero polynomial equation with coefficients in that field.  We
write $[u^m]F(u)$ for the coefficient of $u^m$ in $F(u)$ and use
$\mathbb N_0=\{0,1,2,\ldots\}$.

\subsection{Main results}

Let $(Y,i)$ be a locally finite edge-indexed graph, let $\OrE(Y)$ be its set of
oriented edges, and write $o(e)$, $t(e)$ and $\bar e$ for the origin, terminus
and reverse of an oriented edge $e$.  Let $\one_A$ denote the indicator of a
condition $A$.  The weighted non-backtracking operator $B$, indexed by
$\OrE(Y)$, has transition weights
\[
 B_{e,f}=\begin{cases}
 i(f)-\one_{\{f=\bar e\}},&t(e)=o(f),\\
 0,&\text{otherwise}.
 \end{cases}
\]
This is the quotient-level multiplicity of non-backtracking lifts in the
Bass--Serre tree.

The obstruction theorem in \cref{thm:obstruction} shows that infinitely many
length-four cycles make $\Tr(B^4)$ diverge.  For a finite observation set
$E\subset\OrE(Y)$, however, the first-return matrix $\cF_E(u)$ is
coefficientwise finite and
\[
 Z_E(u)=\det(I-\cF_E(u))^{-1}
\]
is the Euler product over primitive positive weighted cycles meeting $E$; see
\cref{sec:local}.

For the one-sided comb model, a single vertical cuspidal ray can be
eliminated exactly and the remaining two-state horizontal problem can be
solved algebraically.  The next theorem states the outcome; the successive
quantities in the formula are introduced and computed in
\cref{sec:explicit}.

\begin{theorem}[Algebraic root local zeta]
\label{thm:intro-algebraic}
For this model with even index parameter $q\ge2$, the local zeta based at
the root horizontal edge $a_0$ is
\[
 Z_{a_0}(u)=\frac1{1-K(u)x(u)}.
\]
Here $x$ is the branch through $x(0)=0$ of the quadratic first-passage
equation obtained from the two-state stationary cell, and $K$ is the explicit
root correction; see \cref{eq:xquadratic,eq:xclosed,eq:Kroot}.  In particular
$Z_{a_0}$ is algebraic of degree at most two over the rational-function field
$\mathbb Q(q,u)$.  Its Taylor series has exact radius $(q+1)^{-1}$, and its
two dominant singularities are simple poles at $u=\pm(q+1)^{-1}$.
\end{theorem}

Because the model is bipartite, recurrence is studied on the communicating
component of $a_0$ for $B^2$, with the locally constant potential
$\phi_q(e,f)=\log\!\bigl((B^2)_{e,f}\bigr)$.  Precise definitions are given in
\cref{def:quotient-edge-shift,def:squared-root-component,def:multiplicity-potential}.

\begin{theorem}[Pressure gap and strong positive recurrence]
\label{thm:intro-pressure-gap}
The Gurevich pressure and pressure at infinity of the two-step multiplicity
potential are
\[
 P_G(\phi_q)=2\log(q+1),
 \qquad
 P_\infty(\phi_q)=\log(4q).
\]
In particular $P_\infty(\phi_q)<P_G(\phi_q)$, so $\phi_q$ is strongly
positive recurrent.  Moreover, put $v=u^2$ and define $N_{a_0}(n)$ by
\[
 v\frac{d}{dv}\log Z_{a_0}(\sqrt v)
 =\sum_{n\ge1}N_{a_0}(n)v^n.
\]
There is a constant $\kappa_q<(q+1)^2$ such that
\[
 N_{a_0}(n)=(q+1)^{2n}+O(\kappa_q^n).
\]
\end{theorem}
The same Schur scheme proves algebraicity for eventually stationary tails
with finite boundary data and algebraic cell scattering; see
\cref{thm:stationary-algebraicity}.

Finally let $h(e)=d_{Y_q}(o(e),v_0)$ be the graph distance from the root and
let $M_\rho$ be multiplication by $\rho^{h(e)}$, $0<\rho<1$, on the Hilbert
space $\ell^2(\OrE(Y_q))$ of square-summable edge-state functions.  For a
cycle $C=(e_0,\ldots,e_{m-1})$, put
$H(C)=\sum_{j=0}^{m-1}h(e_j)$.  We regularize by $T_\rho=BM_\rho$.

For every $0<\rho<1$, $T_\rho$ is trace class and
\[
 Z(u,\rho)=\det(I-uT_\rho)^{-1}
\]
is the regulator used below.

The periodic tail is stationary under horizontal translation.  Let
$\mathcal P_{\rm cell}$ denote the primitive tail cycles anchored by requiring
minimal top coordinate one, and write $(\ell_P,w_P,\eta_P)$ for their length,
weight and anchored integrated height.
The corresponding anchored cell Euler product is
\[
 Z_{\rm cell}(u)
 =\prod_{P\in\mathcal P_{\rm cell}}(1-w_Pu^{\ell_P})^{-1},
\]
initially as an analytic germ at $u=0$.  When the dependence on the height
function must be displayed, we write $Z_{\rm ren}^{h}$ for the renormalized
finite part constructed from $h$.

\begin{theorem}[Complete renormalized limit and height covariance]
\label{thm:intro-renormalized}
For $|u|<(q+1)^{-1}$ the normally convergent
series
\[
 \mathscr P(u)=\sum_{P\in\mathcal P_{\rm cell}}
 \frac1{\ell_P}\Li(w_Pu^{\ell_P})
\]
defines the complete integrated-pressure counterterm, and
\[
 Z_{\rm ren}(u)=\lim_{\rho\uparrow1}
 \exp\!\left(-\frac{\mathscr P(u)}{1-\rho}\right)Z(u,\rho)
\]
exists locally uniformly.  It is given by the convergent product
\[
 Z_{\rm root}(u)
 \prod_{P\in\mathcal P_{\rm cell}}
 (1-w_Pu^{\ell_P})^{\eta_P/\ell_P-1/2}
 \exp\!\left[-\frac{\Li(w_Pu^{\ell_P})}{2\ell_P}\right].
\]
All logarithms and fractional powers in this formula are the analytic branches
equal to zero and one, respectively, at $u=0$.
Here $Z_{\rm root}(u)$ is the algebraic Euler product over primitive cycles
meeting the root vertex $v_0$; it generally differs from the edge-local zeta
$Z_{a_0}$.  Its explicit factorization is given in
\eqref{eq:rootfactor}.  The same conclusion holds for every stationary
translated-orbit system satisfying the exponential orbit bounds of
\cref{def:translated-orbit-system}.  If the height is replaced by $h+c$,
$c\in\mathbb Z_{\ge0}$, then
\[
 Z_{\rm ren}^{h+c}(u)=Z_{\rm cell}(u)^{-c}Z_{\rm ren}^{h}(u),
\]
whereas a finite-support change of height leaves $Z_{\rm ren}$ unchanged.
\end{theorem}

Here \emph{complete} means that the leading Abel divergence is removed for
every translated primitive tail family; height covariance records the
remaining dependence on the normalization of height.

\subsection{Relation to previous work: theorem-by-theorem positioning}

We distinguish the inherited mechanisms from the model-specific conclusions.

\paragraph{Global obstruction and the local theorem
(\cref{thm:obstruction,thm:intro-algebraic}).}
Global and periodic infinite-graph determinant frameworks include
\cite{BowenLanford1970,Bass1992,ClairMokhtariSharghi2001,
GuidoIsolaLapidus2008Periodic,GuidoIsolaLapidus2008Approximation,Deitmar2015},
while Chinta--Jorgenson--Karlsson and Kousaka localize closed-path counts at a
vertex \cite{ChintaJorgensonKarlsson2015,Kousaka2023}.  Thus
\cref{thm:intro-algebraic} does not claim that localization itself is new.  Its
observation set consists of directed edge states, its determinant is induced
by first returns, and its weights are graph-of-groups non-backtracking
multiplicities.  The new conclusions are the coefficientwise length-four
obstruction and the exact stationary Schur computation, including the
quadratic formula, convergence radius and dominant poles.  The finite-core,
finitely cusped theories of Deitmar--Kang and Hong--Kwon do not encounter this
infinite fixed-length multiplicity \cite{DeitmarKang2018,HongKwon2025}.
Deitmar's weighted prime-geodesic asymptotics concern a different, globally
defined setting \cite{Deitmar2023}; they do not supply a coefficientwise
global zeta for the present infinitely repeated length-four cycles.

\paragraph{Pressure theorem (\cref{thm:intro-pressure-gap}).}
The thermodynamic formalism and the recurrence criterion
$P_\infty<P_G$ are taken from
\cite{GurevichSavchenko1998,Sarig1999,Sarig2001,SarigPhase2001,
RuhrSarig2022,Velozo2026}.  The new part of
\cref{thm:intro-pressure-gap} is the exact
evaluation of both pressures for the graph-of-groups multiplicity potential;
the poles of the algebraic local zeta then give the orbit asymptotic with an
explicit spectral scale.  The tree codings and mixing theory provide the
geometric background, not these calculations
\cite{BroisePaulinCoding2007,BroisePaulinModular2007,
BroiseParkkonenPaulin2016,BroiseParkkonenPaulin2019,
BroiseParkkonenPaulin2021}.

\paragraph{Renormalized theorem (\cref{thm:intro-renormalized}).}
Guido--Isola--Lapidus obtain a periodic-graph zeta as a normalized amenable
finite-subgraph limit \cite{GuidoIsolaLapidus2008Approximation}.  Here the
one-sided quotient has a root boundary and its unnormalized Euler product is
not a formal series.  We instead use a height-Abel regulator and subtract the
integrated-pressure divergence of every translated primitive family.  The new
points are uniformity over primitive lengths and weights, the explicit
dilogarithmic counterterm, local uniformity of the finite part and height
covariance.  Fixed shifted-factorial asymptotics \cite{McIntosh1999} treat one
family but do not justify their sum.  This limit is therefore neither analytic
continuation of a global zeta nor the periodic normalized determinant of
\cite{GuidoIsolaLapidus2008Periodic,GuidoIsolaLapidus2008Approximation}.

\paragraph{Stationary extension and finite-volume comparison
(\cref{thm:stationary-algebraicity,thm:finite-renormalized}).}
Schur complements and transfer matrices are standard; the contribution of
\cref{thm:stationary-algebraicity} is a finite-boundary criterion for
eventually stationary cuspidal tails.  Similarly,
\cref{thm:finite-renormalized} is a one-sided, boundary-corrected comparison
theorem, not a generalization of \cite{GuidoIsolaLapidus2008Approximation}.
It identifies the relative finite factor and compares sharp-cutoff and
height-Abel finite parts.

\section{The one-sided comb model and the global obstruction}
\label{sec:prelim}

\subsection{Edge-indexed graphs and weighted transitions}

We first specify the quotient data without assuming that the reader is
familiar with graphs of groups.  For a graph $Y$, write $VY$ for its vertex
set and $\OrE(Y)$ for the set of oriented edges.  Each geometric edge gives
two oriented edges, $e$ and its reverse $\bar e$.  Their origins and termini
are denoted by $o(e)$ and $t(e)$, so that
$o(\bar e)=t(e)$ and $t(\bar e)=o(e)$.

\begin{definition}[Edge-indexed graph]\label{def:edge-indexed-graph}
An \emph{edge-indexed graph} is a pair $(Y,i)$ consisting of a graph $Y$ and
a function
\[
 i:\OrE(Y)\longrightarrow \mathbb Z_{\ge1}.
\]
Thus an index is attached to an \emph{oriented} edge.  In particular,
$i(e)$ and $i(\bar e)$ need not be equal.  We call $(Y,i)$ \emph{locally
finite} when
\begin{equation}\label{eq:local-finite-indexed}
 \sum_{o(e)=v}i(e)<\infty
 \qquad(v\in VY).
\end{equation}
\end{definition}

The source of these indices in the present paper is Bass--Serre theory.  A
graph of groups over $Y$ assigns a group $G_v$ to each vertex, a group
$G_e=G_{\bar e}$ to each underlying geometric edge, and an attaching
monomorphism $\phi_e:G_e\hookrightarrow G_{o(e)}$ at each end of that edge.
We refer to \cite{Serre2003,BassLubotzky2001} for the Bass--Serre background.
These data canonically determine the edge index
\[
 i(e)=[G_{o(e)}:\phi_e(G_e)].
\]
We assume throughout that these indices are finite and satisfy
\eqref{eq:local-finite-indexed}.  For a quotient of a locally finite
Bass--Serre tree, the sum in \eqref{eq:local-finite-indexed} is exactly the
degree of any lift of $v$.  After this point, the weighted path calculations
use only the pair $(Y,i)$; the graph-of-groups realization explains the
multiplicities geometrically.

Define the transition multiplicity
\begin{equation}\label{eq:transition}
 b(e,f)=
 \begin{cases}
 i(f)-1,&t(e)=o(f)\text{ and }f=\bar e,\\
 i(f),&t(e)=o(f)\text{ and }f\ne\bar e,\\
 0,&t(e)\ne o(f).
 \end{cases}
\end{equation}
The subtraction in the first line removes the unique lift that immediately
reverses the preceding lifted edge.  Quotient backtracking is therefore
allowed precisely when the corresponding index supplies another lift.

Write $B=(b(e,f))_{e,f\in\OrE(Y)}$.  The matrix $B$ records both which
quotient transitions are possible and how many non-backtracking lifts each
transition has.

\begin{definition}[Quotient edge shift]\label{def:quotient-edge-shift}
The \emph{two-sided quotient edge shift} associated with $B$ is
\[
 \Sigma_B=
 \bigl\{x=(e_j)_{j\in\mathbb Z}\in\OrE(Y)^{\mathbb Z}:
 b(e_j,e_{j+1})>0\text{ for every }j\in\mathbb Z\bigr\},
\]
with the left shift
\[
 \sigma(x)_j=e_{j+1}.
\]
Its \emph{one-sided presentation} is
\[
 \Sigma_B^+=
 \bigl\{(e_j)_{j\in\mathbb N_0}:b(e_j,e_{j+1})>0
 \text{ for every }j\ge0\bigr\},
\]
with the same left-shift rule.  The two-sided system is the natural extension
of the one-sided system.  We use the former for invertible geodesic dynamics
and the latter for transfer operators, first returns and strong positive
recurrence.

The shift space $\Sigma_B$ remembers only the support condition $b(e,f)>0$;
the positive integer $b(e,f)$ is the weight of that transition.  Equivalently,
one may replace a transition of weight $m$ by $m$ parallel symbolic branches.
This expanded presentation retains the multiplicities while having the same
weighted path sums as powers of $B$.
\end{definition}

\begin{definition}\label{def:pathweight}
For an admissible closed edge sequence
$C=(e_0,e_1,\ldots,e_{m-1})$, meaning that
$t(e_j)=o(e_{j+1})$ with indices read modulo $m$, set
\[
 \ell(C)=m,
 \qquad
 W(C)=\prod_{j=0}^{m-1}b(e_j,e_{j+1}).
\]
We call $C$ \emph{positive} if $W(C)>0$.  Cyclic rotations are identified and
their equivalence class is denoted by $[C]$.  The cycle $C$ is
\emph{primitive} if it cannot be written as $D^k$ for a closed cycle $D$ and
an integer $k\ge2$.
\end{definition}

The matrix $B$ is the weighted non-backtracking operator.  For this
nonnegative matrix we use the possibly infinite diagonal
sum
\[
 \Tr(B^m):=\sum_{e\in\OrE(Y)}(B^m)_{e,e}\in[0,\infty].
\]
On a finite graph one has \cite{Bass1992}
\begin{equation}\label{eq:finitezeta}
 Z_Y(u)=\prod_{[C]\ \mathrm{primitive}}
 (1-W(C)u^{\ell(C)})^{-1}
 =\det(I-uB)^{-1}.
\end{equation}
The same formula can hold for selected infinite weighted graphs, but it requires
a finiteness or trace condition \cite{Deitmar2015,DeitmarKang2018}.

\Needspace{0.34\textheight}
\subsection{The local--regularized--renormalized hierarchy}
\Cref{tab:zeta-hierarchy} summarizes the four levels used throughout the
paper and the distinct finiteness mechanism at each level.

\begin{table}[H]
\centering
\caption{The global, local, regularized and renormalized levels.}
\label{tab:zeta-hierarchy}
\begin{tabular}{>{\raggedright\arraybackslash}p{0.20\textwidth}
                >{\raggedright\arraybackslash}p{0.33\textwidth}
                >{\raggedright\arraybackslash}p{0.35\textwidth}}
\toprule
Object & Periodic data retained & Finiteness mechanism\\
\midrule
Global zeta & all primitive weighted cycles & none; it fails for $Y_q$\\
Local zeta $Z_E$ & primitive cycles meeting a fixed finite $E$ & finite
propagation from $E$\\
Regularized family $Z(u,\rho)$ & all primitive cycles with a regulator &
height damping $\rho^{H(C)}$, $0<\rho<1$\\
Renormalized finite part $Z_{\rm ren}$ & all primitive orbit families after
bulk subtraction & remove $\mathscr P(u)/(1-\rho)$ and let $\rho\uparrow1$\\
\bottomrule
\end{tabular}
\end{table}

\subsection{The infinitely cusped one-sided comb model}
\label{sec:comb}

Fix an even integer $q\ge2$ and put
\[
 r=\frac q2+1.
\]
Let $Y_q$ have top vertices $v_0,v_1,\ldots$, horizontal edges
$a_n:v_n\to v_{n+1}$, and, for each $n\ge0$, a vertical ray with vertices
$w_{n,k}$, $k\ge1$, and downward edges
\[
 b_{n,0}:v_n\to w_{n,1},
 \qquad
 b_{n,k}:w_{n,k}\to w_{n,k+1}\quad(k\ge1).
\]

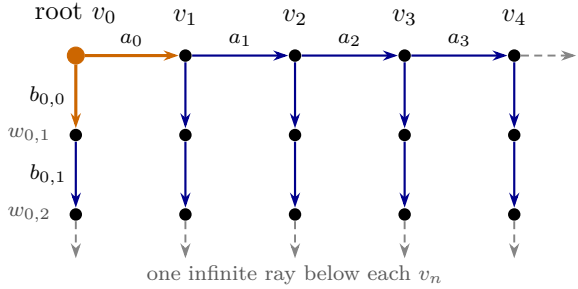
\begin{figure}[H]
\centering
\begin{tikzpicture}[x=1cm,y=1cm,
  >={Stealth[length=2.0mm,width=1.35mm]},
  vertex/.style={circle,fill=black,inner sep=1.65pt},
  root vertex/.style={circle,fill=orange!80!black,draw=orange!80!black,
                     inner sep=2.25pt},
  stationary edge/.style={->,draw=blue!55!black,line width=0.82pt},
  exceptional edge/.style={->,draw=orange!80!black,line width=1.18pt},
  continuation/.style={->,densely dashed,draw=black!48,line width=0.72pt},
  edge name/.style={font=\scriptsize,fill=white,inner sep=1.1pt,text=black},
  vertex name/.style={font=\small,text=black},
  legend/.style={draw=black!32,fill=black!2,rounded corners=2pt,
                 inner sep=7pt,text width=6.45cm,align=left,font=\small}]

  \node[root vertex] (v0) at (0,0) {};
  \foreach \n in {1,...,4}{
    \node[vertex] (v\n) at ({1.45*\n},0) {};
  }
  \foreach \n in {0,...,4}{
    \node[vertex] (w\n1) at ({1.45*\n},-1.05) {};
    \node[vertex] (w\n2) at ({1.45*\n},-2.10) {};
  }

  \draw[exceptional edge] (v0)--node[edge name,above=2pt] {$a_0$} (v1);
  \draw[stationary edge] (v1)--node[edge name,above=2pt] {$a_1$} (v2);
  \draw[stationary edge] (v2)--node[edge name,above=2pt] {$a_2$} (v3);
  \draw[stationary edge] (v3)--node[edge name,above=2pt] {$a_3$} (v4);
  \draw[continuation] (v4)--+(0.82,0);

  \draw[exceptional edge] (v0)--node[edge name,left=2pt] {$b_{0,0}$} (w01);
  \draw[stationary edge] (w01)--node[edge name,left=2pt] {$b_{0,1}$} (w02);
  \foreach \n in {1,...,4}{
    \draw[stationary edge] (v\n)--(w\n1);
    \draw[stationary edge] (w\n1)--(w\n2);
  }
  \foreach \n in {0,...,4}{
    \draw[continuation] (w\n2)--+(0,-0.58);
  }

  \node[vertex name,above=5pt of v0] {root $v_0$};
  \foreach \n in {1,...,4}{
    \node[vertex name,above=5pt of v\n] {$v_{\n}$};
  }
  \node[font=\scriptsize,text=black!70,left=3pt of w01] {$w_{0,1}$};
  \node[font=\scriptsize,text=black!70,left=3pt of w02] {$w_{0,2}$};
  \node[font=\scriptsize,text=black!65] at (2.90,-2.90)
       {one infinite ray below each $v_n$};

  \node[legend,anchor=north west] at (7.02,0.58) {%
    {\bfseries How to read the indices}\par\smallskip
    Every solid arrow $e$ points away from $v_0$.  Its ordered pair is
    $(i(e),i(\bar e))$, with the reverse index second.\par\medskip
    \(\begin{array}{rcl}
      a_0 &\longmapsto& (r,q),\\
      a_n\ (n\ge1) &\longmapsto& (1,q),\\
      b_{0,0} &\longmapsto& (r,q+1),\\
      b_{n,k}\ ((n,k)\ne(0,0)) &\longmapsto& (1,q+1).
    \end{array}\)\par\medskip
    Here $r=q/2+1$.  The two thicker orange arrows are the exceptional root
    edges; the thinner blue arrows obey the uniform non-root rules.
  };
\end{tikzpicture}
\caption{The one-sided comb quotient $Y_q$ as a rooted edge-indexed graph.
Every solid arrow is oriented away from $v_0$, dashed arrows indicate
continuation, and the box records the ordered pair $(i(e),i(\bar e))$.
Colour and line weight separate the two exceptional root edges from the
non-root tail, but the labels make the distinction independent of colour.}
\label{fig:comb}
\end{figure}

We call $n$ the \emph{top coordinate} of $v_n$.  For $n\ge1$, the $n$th
\emph{stationary cell} means the local transition data at $v_n$, together with
its attached vertical ray and its left and right horizontal interfaces.  The
data at $v_0$ form the exceptional root cell.
Thus the underlying unindexed graph is the rooted, one-sided comb: a
horizontal ray with one vertical ray attached at every vertex.  The term is
descriptive; unlike the classical bilateral comb lattice on $\mathbb Z^2$,
the present model is one-sided and carries essential edge indices
\cite{KrishnapurPeres2004}.
The indices are
\begin{align}
 i(a_0)&=r,&i(\bar a_0)&=q,
 &i(b_{0,0})&=r,&i(\bar b_{0,0})&=q+1,\label{eq:rootindices}\\
 i(a_n)&=1,&i(\bar a_n)&=q &&(n\ge1),\label{eq:horizontalindices}\\
 i(b_{n,k})&=1,&i(\bar b_{n,k})&=q+1
 &&((n,k)\ne(0,0)).\label{eq:verticalindices}
\end{align}
This is the edge-indexed form of the quadratic-growth example in
\cite[pp.~305--306]{BroiseParkkonenPaulin2021}.

\begin{proposition}[Graph-of-groups realization and finite covolume]
\label{prop:lattice-realization}
The indexed graph $(Y_q,i)$ is realized by a graph of finite groups whose
fundamental group $\Gamma_q$ acts discretely on the $(q+2)$-regular
Bass--Serre tree $X_q$, with quotient $\Gamma_q\backslash X_q=Y_q$.  With the
standard vertex-stabilizer normalization, its covolume is finite; more
precisely,
\begin{equation}\label{eq:covolume}
 \sum_{x\in VY_q}\frac1{|G_x|}
 =\frac1r+\frac2{q-1}<\infty.
\end{equation}
Hence $\Gamma_q$ is a nonuniform tree lattice.
\end{proposition}

\begin{proof}
A concrete realization of the same edge-indexed graph, consistent with
\cite[pp.~305--306]{BroiseParkkonenPaulin2021}, is as follows.  Set
\[
 G_{v_0}=\mathbb Z/r\mathbb Z,\qquad
 G_{v_n}=\mathbb Z/q^n\mathbb Z\quad(n\ge1),
\]
and
\[
 G_{w_{n,k}}
 =\mathbb Z/q^n\mathbb Z\times
  \mathbb Z/(q+1)^k\mathbb Z
 \qquad(n\ge0, k\ge1),
\]
where a factor with exponent zero is trivial.  Give the two edges issuing
from $v_0$ trivial edge groups.  For every other edge directed away from
$v_0$, take the edge group to be its origin group and embed it identically at
the origin and as the evident subgroup at the terminus.  The resulting
indices are exactly \eqref{eq:rootindices}--\eqref{eq:verticalindices}.

All vertex stabilizers are finite, so the Bass--Serre action is discrete.
Moreover,
\begin{align*}
 \sum_{x\in VY_q}\frac1{|G_x|}
 &=\frac1r+\sum_{n\ge1}\frac1{q^n}
   +\sum_{n\ge0}\sum_{k\ge1}
     \frac1{q^n(q+1)^k}\\
 &=\frac1r+\frac1{q-1}+\frac1{q-1},
\end{align*}
which proves \eqref{eq:covolume}.  Finally, the universal-cover degree is
$q+2$ by \eqref{eq:degree} below.
\end{proof}

\begin{remark}[Meaning of cusp]
\label{rem:cusp-convention}
In this paper a \emph{vertical cusp ray} means the quotient ray
$v_n,w_{n,1},w_{n,2},\ldots$.  Beyond its initial vertex it has outward index
$1$, inward index $q+1$, and exponentially increasing stabilizer orders
$q^n(q+1)^k$; its stabilizer mass is therefore summable.  This is the standard
ray-of-finite-volume feature relevant to the present graph-of-groups model.
Since there is one such ray for every $n\ge0$, the lattice has infinitely many
vertical cusps.  No assertion of geometric finiteness with finitely many cusp
orbits is intended.
\end{remark}

\subsection{Coefficientwise failure of the global zeta}
\label{sec:globalfail}

We first isolate the general mechanism.

\begin{theorem}[Infinite length-four obstruction]\label{thm:obstruction}
Let $(Y,i)$ be a locally finite edge-indexed graph.  Suppose that there are
infinitely many triples $(v_n,e_n,f_n)$ such that
\[
 o(e_n)=o(f_n)=v_n,\quad e_n\ne f_n,
\]
the unoriented edge pairs $\{e_n,\bar e_n\}$ and $\{f_n,\bar f_n\}$ determine
distinct cycles as $n$ varies, and
\[
 i(e_n)=i(f_n)=1,
 \qquad i(\bar e_n)\ge2,
 \qquad i(\bar f_n)\ge2.
\]
Then the cycles
\[
 C_n=(e_n,\bar e_n,f_n,\bar f_n)
\]
are positive and primitive, and
\[
 W(C_n)=\bigl(i(\bar e_n)-1\bigr)
         \bigl(i(\bar f_n)-1\bigr)>0.
\]
Consequently the coefficient of $u^4$ in the logarithm of the global Euler
product is infinite.  Equivalently, $\Tr(B^4)=\infty$.
\end{theorem}

\begin{proof}
Using \eqref{eq:transition} around $C_n$ gives
\begin{align*}
 W(C_n)
 &=b(e_n,\bar e_n)b(\bar e_n,f_n)
   b(f_n,\bar f_n)b(\bar f_n,e_n)\\
 &=\bigl(i(\bar e_n)-1\bigr)i(f_n)
   \bigl(i(\bar f_n)-1\bigr)i(e_n),
\end{align*}
which is the asserted positive number.  The cycle uses two distinct
unoriented edges, so it is not a square of a length-two cycle and is primitive.
Each $C_n$ supplies four contributions to $\Tr(B^4)$, one for each choice of
starting edge around the cycle.
Since all contributions are nonnegative, the trace diverges.
\end{proof}

\begin{corollary}[The one-sided comb example]\label{cor:combdiv}
For $n\ge1$, let
\[
 C_n=(a_n,\bar a_n,b_{n,0},\bar b_{n,0}).
\]
Then
\begin{equation}\label{eq:Cnweight}
 W(C_n)=(q-1)\cdot1\cdot q\cdot1=q(q-1).
\end{equation}
Thus the global edge-indexed zeta function of $Y_q$ is undefined as a formal
power series.  On the other hand, the ordinary unweighted Ihara zeta of the
underlying one-sided comb graph is $1$.
\end{corollary}

We next record the regularity and growth features of the quotient before
passing to its two-step dynamics.

At the root, the two outgoing indices sum to $2r=q+2$.  At a top vertex
$v_n$, $n\ge1$, the outgoing indices are $q,1,1$, and at a vertical vertex
they are $q+1,1$.  Thus
\begin{equation}\label{eq:degree}
 \sum_{o(e)=v}i(e)=q+2
 \qquad(v\in VY_q),
\end{equation}
so the Bass--Serre universal cover is $(q+2)$-regular.

The number of vertices at root distance $m$ is $m+1$: it consists of $v_m$
and the vertices $w_{n,m-n}$ for $0\le n<m$.  Therefore
\[
 |B_{Y_q}(v_0,N)|=\sum_{m=0}^N(m+1)
 =\frac{(N+1)(N+2)}2,
\]
where $B_{Y_q}(v_0,N)=\{x\in VY_q:d_{Y_q}(x,v_0)\le N\}$ is the closed
graph-distance ball.  This explains the quadratic growth.

\begin{remark}[Period two]\label{rem:periodtwo}
The underlying graph is bipartite.  If
$\varepsilon(e)=(-1)^{d_{Y_q}(o(e),v_0)}$, then every nonzero transition satisfies
$\varepsilon(f)=-\varepsilon(e)$.  Hence all periodic edge paths have even
length.  Analytic mixing statements for the time-one edge shift must therefore
be formulated either for its square on one parity component or after adding an
aperiodic finite core.  In particular all local series below are even series in
$u$.
\end{remark}

The next definition makes precise the two-step object used in the recurrence
and counting statements.  It also prevents a possible ambiguity: the square
refers to composition of two transitions, not to entrywise squaring of $B$.

\begin{definition}[Squared root component]
\label{def:squared-root-component}
The \emph{two-step transition matrix} is the usual matrix product $B^2$.
Its communicating class containing the root edge $a_0$ is
\[
 \cC_{a_0}=
 \left\{e\in\OrE(Y_q):
 \begin{array}{l}
 (B^{2m})_{a_0,e}>0\text{ for some }m\ge0,\\
 (B^{2n})_{e,a_0}>0\text{ for some }n\ge0
 \end{array}
 \right\}.
\]
Restrict the two-step weights to this class,
\[
 B^{(2)}_{a_0}
 =\bigl((B^2)_{e,f}\bigr)_{e,f\in\cC_{a_0}},
\]
and define
\[
 \Sigma^{(2)}_{a_0}
 =\bigl\{(e_j)_{j\in\mathbb Z}\in\cC_{a_0}^{\mathbb Z}:
 (B^2)_{e_j,e_{j+1}}>0\text{ for every }j\in\mathbb Z\bigr\}.
\]
The corresponding one-sided shift is
\[
 \Sigma^{(2),+}_{a_0}
 =\bigl\{(e_j)_{j\in\mathbb N_0}\in\cC_{a_0}^{\mathbb N_0}:
 (B^2)_{e_j,e_{j+1}}>0\text{ for every }j\ge0\bigr\}.
\]
With the left shift $\sigma_2((e_j)_j)=(e_{j+1})_j$, these are respectively
the two-sided squared root component and its one-sided presentation.  The
former is the natural extension of the latter.
\end{definition}

Both shifts are irreducible by construction.  The positive two-step first
return in \cref{eq:f2} makes the one-sided presentation aperiodic and hence
topologically mixing, while local finiteness of $Y_q$ makes it locally compact.
All recurrence and pressure terminology below refers to this one-sided
presentation together with the multiplicity potential defined next.

\subsection{The multiplicity potential and pressure at infinity}
\label{sec:weighted-thermo-definitions}

The support shift alone forgets the number of non-backtracking lifts.  We
therefore formulate recurrence for a potential that retains exactly the
weights used by the zeta functions.

For a potential $\phi$ on the one-sided shift, write
\begin{equation}\label{eq:birkhoff-variation-conventions}
 S_n\phi(x)=\sum_{j=0}^{n-1}\phi(\sigma_2^j x),
 \qquad
 \operatorname{var}_n(\phi)=
 \sup\bigl\{|\phi(x)-\phi(y)|:x,y\in\Sigma^{(2),+}_{a_0},
 x_j=y_j\text{ for }0\le j<n\bigr\}.
\end{equation}
The potential has \emph{summable variations} if
$\sum_{n\ge1}\operatorname{var}_n(\phi)<\infty$, and it is \emph{weakly
H\"older} if $\operatorname{var}_n(\phi)\le C\theta^n$ for some $C>0$ and
$0<\theta<1$.

\begin{definition}[Two-step multiplicity potential]
\label{def:multiplicity-potential}
On the one-sided squared root component $\Sigma^{(2),+}_{a_0}$ define
\begin{equation}\label{eq:multiplicity-potential}
 \phi_q(x)=\log\!\bigl((B^2)_{x_0,x_1}\bigr).
\end{equation}
Equivalently, on an allowed transition $e\to f$ of the squared component we
write $\phi_q(e,f)=\log\!\bigl((B^2)_{e,f}\bigr)$.
\end{definition}

Every allowed entry of
$B^2$ is a positive integer at most $(q+1)^2$, so
$0\le\phi_q\le2\log(q+1)$.  Moreover, $\phi_q$ depends only on $(x_0,x_1)$;
hence $\operatorname{var}_n(\phi_q)=0$ for $n\ge2$.  In particular it is
bounded, weakly H\"older and has summable variations.  For every periodic
point $x$ of
period $n$,
\begin{equation}\label{eq:weighted-periodic-potential}
 \exp(S_n\phi_q(x))
 =\prod_{j=0}^{n-1}(B^2)_{x_j,x_{j+1}},
 \qquad x_n=x_0.
\end{equation}
Thus its periodic partition sums are precisely the two-step weighted path
sums used below.

For a one-sided transitive countable Markov shift $(\Sigma,\sigma)$, a state
$a$ and a potential $\phi$, set
\[
 \operatorname{Per}_a(n)
 =\{x\in\Sigma:\sigma^n x=x,\ x_0=a\}.
\]
Thus $\operatorname{Per}_a(n)$ consists of points fixed by $\sigma^n$; their
least period is allowed to divide $n$.  The Gurevich pressure is
\begin{equation}\label{eq:weighted-gurevich-pressure}
 P_G(\phi)=\limsup_{n\to\infty}\frac1n
 \log\sum_{x\in\operatorname{Per}_a(n)}e^{S_n\phi(x)}.
\end{equation}
For a transitive shift and a summable-variation potential this does not depend
on $a$ \cite{Sarig1999}.

We also recall the periodic-orbit definition of pressure at infinity.  If
$F$ is a finite set of states and $M\ge1$, put
\begin{align}
 \operatorname{Per}_a(n;F,M)
 &=\left\{x\in\operatorname{Per}_a(n):
 \#\{0\le j<n:x_j\in F\}\le\frac nM\right\},
 \label{eq:escaping-periodic-set}\\
 P_\infty(\phi;a,F,M)
 &=\limsup_{n\to\infty}\frac1n
 \log\sum_{x\in\operatorname{Per}_a(n;F,M)}e^{S_n\phi(x)},
 \label{eq:escaping-periodic-pressure}
\end{align}
where the logarithm of an empty sum is $-\infty$.  Define
\begin{equation}\label{eq:pressure-at-infinity-definition}
 P_\infty(\phi)
 =\inf_{F\Subset\cC_{a_0}}\ \inf_{M\ge1}
 P_\infty(\phi;a,F,M).
\end{equation}
Here $F\Subset\cC_{a_0}$ means that $F$ is finite.  The value is independent
of the base state for the present transitive system.  This is the topological
pressure at infinity used in \cite{RuhrSarig2022,Velozo2026}; it measures the
weighted exponential complexity of periodic trajectories that spend an
asymptotically negligible proportion of time in every fixed finite core.
Equivalently, if
$F_1\subset F_2\subset\cdots$ is any exhaustion of $\cC_{a_0}$ by finite
sets, then monotonicity in $F$ and cofinality of the exhaustion give
\[
 P_\infty(\phi)
 =\inf_{j\ge1}\inf_{M\ge1}P_\infty(\phi;a,F_j,M).
\]
Thus the formulation above agrees with the exhaustion-based definition in
\cite[Section~3.1]{Velozo2026}.

Accordingly, every assertion of strong positive recurrence below concerns
the pair $(\Sigma^{(2),+}_{a_0},\phi_q)$.  For a bounded weakly H\"older
potential of finite pressure, the strict inequality
$P_\infty(\phi)<P_G(\phi)$ is equivalent to strong positive recurrence
\cite[Theorem~8.2]{RuhrSarig2022}; see also \cite{Velozo2026}.  We use this
equivalent pressure-gap criterion whenever we call the present potential
strongly positive recurrent.

\section{First-return local determinants}
\label{sec:local}

Vertex-based generalized Ihara formulas for infinite regular graphs appear in
\cite{ChintaJorgensonKarlsson2015}.  The definition below is adapted to
edge-indexed quotient geodesics and uses first returns so that repeated visits
to the observation set are handled by a finite induced determinant.

\subsection{First-return matrices and local determinants}

Let $E\subset\OrE(Y)$ be finite.  Write $P=P_E$ for coordinate projection onto
$\mathbb C^E$, let $I$ be the identity on the full edge-state space, and put
$Q=I-P$.  Let $I_E$ denote the identity on $\mathbb C^E$.  A first-return path
of length $n$ from $e\in E$ to
$f\in E$ starts at $e$, ends at $f$, and has all intermediate states in
$E^c$.  Here ``length $n$'' means $n$ transitions, so the path is a state
sequence $(e_0,e_1,\ldots,e_n)$ with $e_0=e$ and $e_n=f$.  Its weight is the
product of its $n$ transition multiplicities.

\begin{definition}\label{def:firstreturn}
For $n\ge1$, define the $E\times E$ matrix $F_E(n)$ by weighted first-return
counts.  Equivalently,
\begin{align}
 F_E(1)&=PBP,\label{eq:F1}\\
 F_E(n)&=PBQ(QBQ)^{n-2}QBP\qquad(n\ge2).\label{eq:Fn}
\end{align}
The first-return generating matrix and local zeta function are
\begin{equation}\label{eq:localzeta}
 \cF_E(u)=\sum_{n\ge1}F_E(n)u^n,
 \qquad
 Z_E(u)=\det(I_E-\cF_E(u))^{-1}.
\end{equation}
\end{definition}

\begin{proposition}[Formal well-definedness]\label{prop:formalfinite}
If $Y$ is locally finite and every edge index is finite, every coefficient of
$\cF_E(u)$ is finite.  Hence $Z_E(u)\in1+u\mathbb C[[u]]$ is well defined.
\end{proposition}

\begin{proof}
A length-$n$ path starting in the finite set $E$ stays in the finite
graph-distance $n$-neighborhood of the finitely many endpoints of $E$.
Local finiteness gives only finitely many such edge paths, and every path has
finite weight.
The determinant in \eqref{eq:localzeta} is finite-dimensional.
\end{proof}
\subsection{Schur complements and Green functions}

The compressed resolvent $P(I-uB)^{-1}P$ is called the \emph{local Green
matrix} on $E$.  Its $(e,f)$ entry is the generating series of all weighted
paths from $e\in E$ to $f\in E$, with no restriction on intermediate visits
to $E$.  The theorem below says that these unrestricted paths are recovered by
concatenating first-return segments.

\begin{theorem}[Local resolvent identity]\label{thm:schur}
As matrices of formal power series,
\begin{equation}\label{eq:schur}
 P(I-uB)^{-1}P=(I_E-\cF_E(u))^{-1}.
\end{equation}
In particular,
\begin{equation}\label{eq:localdetgreen}
 Z_E(u)=\det\bigl(P(I-uB)^{-1}P\bigr).
\end{equation}
\end{theorem}

\begin{proof}
Relative to $P\oplus Q$, the formal operator $I-uB$ has block form
\[
 \begin{pmatrix}
 I_E-uPBP&-uPBQ\\
 -uQBP&I-uQBQ
 \end{pmatrix}.
\]
The inverse of its lower-right block exists formally.  The Schur complement of
that block is
\[
 I_E-uPBP-u^2PBQ(I-uQBQ)^{-1}QBP=I_E-\cF_E(u),
\]
and block inversion gives \eqref{eq:schur}.
\end{proof}

Thus the infinite complement $E^c$ enters only through a finite Schur
correction.

\subsection{Local Euler product}

Let $\mathcal P_E$ be the set of cyclic-rotation classes $[C]$ of primitive
positive cycles that meet $E$.  The length $\ell(C)$ and weight $W(C)$ do not
depend on the representative of $[C]$.  A cycle meeting $E$ decomposes
uniquely, up to cyclic rotation, into first-return segments between consecutive
visits to $E$.

\begin{theorem}[Local Euler product]\label{thm:localeuler}
The following identity holds formally:
\begin{equation}\label{eq:localeuler}
 Z_E(u)=\prod_{[C]\in\mathcal P_E}
 \bigl(1-W(C)u^{\ell(C)}\bigr)^{-1}.
\end{equation}
\end{theorem}

\begin{proof}
Expanding the finite-dimensional determinant gives
\[
 \log Z_E(u)=\sum_{m\ge1}\frac1m\Tr\bigl(\cF_E(u)^m\bigr).
\]
A monomial in $\Tr(\cF_E(u)^m)$ is a cyclic concatenation of $m$
first-return segments.  Unique first-return decomposition identifies these
monomials with powers of primitive cycles meeting $E$, together with a chosen
starting visit to $E$.  The factor $1/m$ cancels the number of choices of that
starting return, producing the logarithm of the Euler product
\eqref{eq:localeuler}.
\end{proof}

\begin{remark}[Changing the observation set]
Let $E\subset E'$ be finite and put $D=E'\setminus E$.  If
$B^{(E)}$ denotes the transition matrix restricted to the states in $E^c$,
then
\[
 \frac{Z_{E'}^B(u)}{Z_E^B(u)}=Z_D^{B^{(E)}}(u),
\]
where the superscript records the underlying transition matrix.  Indeed, both
sides are the Euler product over primitive cycles that meet $E'$ but avoid
$E$.  Thus the ratio records exactly the additional periodic data seen after
the observation set is enlarged.
\end{remark}

\subsection{Finite exhaustion and coefficient stabilization}

Let $Y^{(N)}$ be an increasing \emph{finite exhaustion}: each $Y^{(N)}$ is a
finite subgraph, $Y^{(N)}\subset Y^{(N+1)}$, and their union is $Y$.  Assume it
contains the $N$-neighborhood of $E$, and let $B_N$ be the corresponding
compressed transition matrix.  Define $\cF_{E,N}$ and $Z_{E,N}$ from $B_N$ by
the same formulas as in \eqref{eq:F1}--\eqref{eq:localzeta}.

Finite propagation gives coefficient stabilization: if $Y^{(N)}$ contains the
graph-distance $m$-neighborhood of the endpoints of $E$, then the coefficients
of $\cF_{E,N}(u)$ and $Z_{E,N}(u)$ through degree $m$ agree with those of
$\cF_E(u)$ and $Z_E(u)$.  Indeed, every contributing path of length at most
$m$ stays in that neighborhood, and the degree-$m$ determinant coefficient
uses only return coefficients of degrees at most $m$.  Hence $Z_E$ is
computable to arbitrary finite order from finite data.

\section{Exact stationary computation and algebraicity}
\label{sec:explicit}

\paragraph{Reading guide to the calculation.}
We successively eliminate a vertical cusp, one stationary cell and the
horizontal half-line; the exceptional root is attached last.  Thus
\[
 \text{vertical cusp }V
 \ \longrightarrow\ \text{cell scattering }\mathsf S
 \ \longrightarrow\ \text{first passages }(x,y)
 \ \longrightarrow\ \text{root factor }K
 \ \longrightarrow\ Z_{a_0}.
\]

\subsection{Vertical-cusp elimination}

Fix $n\ge1$.  Consider a path that starts at $v_n$, descends exactly $k$ edges
along its vertical cusp, turns, and returns to $v_n$ without visiting $v_n$ in
between.  Its quotient edge sequence is
\[
 b_{n,0},b_{n,1},\ldots,b_{n,k-1},
 \bar b_{n,k-1},\ldots,\bar b_{n,1},\bar b_{n,0}.
\]
The downward transitions have weight $1$, the turn has weight
$(q+1)-1=q$, and the remaining $k-1$ upward transitions have weight $q+1$.
Therefore
\begin{equation}\label{eq:excursionweight}
 W_{\rm vert}(k)=q(q+1)^{k-1}.
\end{equation}

\begin{proposition}[Vertical cusp excursion]\label{prop:vertical}
The complete generating function of a vertical excursion at a non-root top
vertex is
\begin{equation}\label{eq:vertical}
 V(u)
 =\sum_{k\ge1}q(q+1)^{k-1}u^{2k}
 =\frac{q u^2}{1-(q+1)u^2}.
\end{equation}
It has first singularities at $u=\pm(q+1)^{-1/2}$.
\end{proposition}

At the exceptional root, both boundary passages have multiplicity $r$; this
modification is incorporated in the factor $K(u)$ in \eqref{eq:Kroot}.

\subsection{Root first-return coefficients}

Choose the one-state observation set
\[
 E=\{a_0\}.
\]
Write
\[
 \cF_{a_0}(u)=\sum_{m\ge1}f_m(q)u^m.
\]
By \cref{rem:periodtwo}, $f_{2m+1}(q)=0$.  The first return of length two is
the root horizontal oscillation
$(a_0,\bar a_0,a_0)$, and hence
\begin{equation}\label{eq:f2}
 f_2(q)=(q-1)(r-1)=\frac{q(q-1)}2.
\end{equation}

The three first-return state sequences of length four are
\begin{align*}
 &(a_0,\bar a_0,b_{0,0},\bar b_{0,0},a_0),\\
 &(a_0,a_1,\bar a_1,\bar a_0,a_0),\\
 &(a_0,b_{1,0},\bar b_{1,0},\bar a_0,a_0).
\end{align*}
Multiplying their transition weights gives
\begin{equation}\label{eq:f4}
 f_4(q)=\frac14q^4+\frac74q^3-\frac12q^2-q.
\end{equation}
Continuing the same finite recursion gives
\begin{equation}\label{eq:f6}
 f_6(q)=\frac18q^6+\frac98q^5+\frac{21}{4}q^4
        +\frac14q^3-2q^2-q.
\end{equation}

For reproducibility, let $g_m(e)$ be the total weight of length-$m$ paths from
$a_0$ to $e$ that avoid $a_0$ at times $1,\ldots,m$.  Thus
$g_0(e)=\one_{\{e=a_0\}}$.  Removing every path as soon as it returns to
$a_0$, use
\begin{align}
 f_{m+1}(q)&=\sum_{e\ne a_0}g_m(e)b(e,a_0),\label{eq:rec-f}\\
 g_{m+1}(f)&=\one_{\{f\ne a_0\}}
 \sum_e g_m(e)b(e,f).\label{eq:rec-p}
\end{align}
Only states within distance $m$ occur, so this is a finite exact calculation at
each stage.

Since $Z_{a_0}(u)=(1-\cF_{a_0}(u))^{-1}$, we obtain
\begin{align}
 Z_{a_0}(u)
 &=1+\frac{q(q-1)}2u^2
 +\left(\frac12q^4+\frac54q^3-\frac14q^2-q\right)u^4\notag\\
 &\quad+\left(\frac12q^6+\frac94q^5+\frac{27}{8}q^4
 -\frac38q^3-q^2-q\right)u^6+O(u^8).
\label{eq:localexpansion}
\end{align}

\begin{table}[ht]
\centering
\caption{The first three nonzero root first-return coefficients.}
\label{tab:first-return-coefficients}
\begin{tabular}{c|rrr}
\toprule
$q$ & $f_2(q)$ & $f_4(q)$ & $f_6(q)$\\
\midrule
$2$ & $1$ & $14$ & $120$\\
$4$ & $6$ & $164$ & $2988$\\
$6$ & $15$ & $678$ & $21360$\\
\bottomrule
\end{tabular}
\end{table}

\paragraph{Analyticity near the origin.}

By \eqref{eq:degree}, the row sum of $B$ is
\[
 \sum_f b(e,f)=q+1.
\]
Consequently the total weight of length-$m$ paths from a fixed state is at
most $(q+1)^m$.  Therefore every finite local first-return matrix converges
absolutely for $|u|<(q+1)^{-1}$, and its local determinant is analytic near the
origin.  For the root observation state this elementary bound is sharp, as the
stationary formula below shows.

\subsection{Operator Schur complement and the algebraic formula}
\label{subsec:stationary-schur}

We now carry out the stationary boundary calculation completely, using the
vertical excursion series $V$ from \eqref{eq:vertical}.
At a non-root top vertex $v_n$, retain the two arrival states
\[
 A_n=a_{n-1}\quad\text{(arrival from the left)},
 \qquad
 B_n=\bar a_n\quad\text{(arrival from the right)}.
\]
We give the block calculation, since retaining the arrival direction is
essential: a scalar vertex recursion would lose the factor subtracted on an
immediate quotient reversal.  Order the boundary arrival channels as
$(A_n,B_n)$ and the horizontal exit channels as (left,right).

\begin{proposition}[One-cell operator Schur complement]
\label{prop:onecell-schur}
In the generating-matrix convention, elimination of all oriented states in
the vertical ray at $v_n$ produces the effective scattering matrix
\begin{equation}\label{eq:sigma-schur}
 \mathsf S(u)=
 u\begin{pmatrix}q-1&1\\ q&0\end{pmatrix}
 +uV(u)\binom{1}{1}\begin{pmatrix}q&1\end{pmatrix}
 =\begin{pmatrix}\alpha&\beta\\ \gamma&\delta\end{pmatrix}.
\end{equation}
Equivalently, its four entries are
\begin{align}
 \alpha(u)&=u\bigl(q-1+qV(u)\bigr),
 &\beta(u)&=u\bigl(1+V(u)\bigr),\label{eq:scatter1}\\
 \gamma(u)&=uq\bigl(1+V(u)\bigr),
 &\delta(u)&=uV(u).\label{eq:scatter2}
\end{align}
\end{proposition}

\begin{proof}
Decompose the one-cell transition operator into horizontal boundary channels
and vertical internal states:
\[
 I-u\mathcal B_n=
 \begin{pmatrix}
  I-uH_{\rm hor}&-u\mathsf E_{\rm in}\\
  -u\mathsf E_{\rm out}&I-uT_v
 \end{pmatrix}.
\]
Here $\mathcal B_n$ is the transition operator for one stationary cell,
$H_{\rm hor}$ is its direct horizontal boundary block, and $T_v$ is the
vertical internal block,
and $\mathsf E_{\rm in},\mathsf E_{\rm out}$ are the boundary-to-interior and
interior-to-boundary coupling blocks.  The notation ${}^{\mathsf T}$ used
below denotes matrix transpose.
Formal block Gaussian elimination gives the boundary Schur complement
\begin{equation}\label{eq:operator-cell-schur}
 I-uH_{\rm hor}-u^2\mathsf E_{\rm in}(I-uT_v)^{-1}\mathsf E_{\rm out}.
\end{equation}
The direct horizontal block is
$uH_{\rm hor}=u\bigl(\begin{smallmatrix}q-1&1\\q&0\end{smallmatrix}\bigr)$:
the left exit following an arrival from the left is the immediate quotient
reverse and therefore has multiplicity $q-1$, while the other three direct
multiplicities are $1,q,0$ in the displayed order.

It remains to evaluate the second term in
\eqref{eq:operator-cell-schur}.  The vertical resolvent entry connecting the
entrance and return channel is precisely the already computed excursion
series $V(u)$.  Both arrival channels enter the vertical cusp ray with the same boundary
factor, giving the column $(1,1)^{\mathsf T}$; after return, the left and right
exit multiplicities are $q$ and $1$, giving the row $(q,1)$.  Hence the
resolvent correction is the rank-one matrix
$uV(u)(1,1)^{\mathsf T}(q,1)$.  No further term occurs: immediately re-entering
the vertical ray after its return has multiplicity $i(b_{n,0})-1=0$.  Adding the
direct and vertical blocks proves \eqref{eq:sigma-schur}.
\end{proof}

Here $\alpha$ and $\beta$ are respectively the left and right exits from
$A_n$, while $\gamma$ and $\delta$ are the left and right exits from $B_n$.
Thus the four scalar series used below are entries of a genuine operator Schur
complement rather than independently postulated path weights.

Let $x(u)$ be the first-passage series from $A_n$ to $B_{n-1}$, and let $y(u)$
be the first-passage series from $B_n$ to $B_{n-1}$.  Translation invariance
for $n\ge1$ gives the stationary Schur system
\begin{equation}\label{eq:xy-system}
 x=\alpha+\beta xy,
 \qquad
 y=\gamma+\delta xy.
\end{equation}
The two non-direct alternatives represented by the product $xy$ are
\[
 A_n\xrightarrow{\beta}A_{n+1}\xrightarrow{x}B_n
     \xrightarrow{y}B_{n-1},
 \qquad
 B_n\xrightarrow{\delta}A_{n+1}\xrightarrow{x}B_n
     \xrightarrow{y}B_{n-1}.
\]
Thus the factor $xy$ records a right step, first passage back to $B_n$, and
then first passage through the original cell.  This ordered product is also
the one that survives in the matrix generalization below.  Eliminating $y$
yields
\begin{equation}\label{eq:xquadratic}
 \delta x^2-\Delta x+\alpha=0,
 \qquad
 \Delta=1+\alpha\delta-\beta\gamma.
\end{equation}
The formal solution with $x(0)=0$ is
\begin{equation}\label{eq:xclosed}
 x(u)=\frac{2\alpha(u)}
 {\Delta(u)+\sqrt{\Delta(u)^2-4\alpha(u)\delta(u)}},
\end{equation}
where the square root is the branch $1+O(u^2)$.  Formula
\eqref{eq:xclosed} is precisely the stationary half-line continued fraction
written as a scalar algebraic Schur complement, with the two arrival directions
retained until the final elimination.

It remains to attach the exceptional root.  Starting in $B_0=\bar a_0$, a
direct return to $a_0$ contributes $u(r-1)$.  Alternatively one enters the
root vertical ray with multiplicity $r$, makes a vertical excursion, and exits
to $a_0$ with multiplicity $r$.  After each such excursion one may re-enter
that same vertical ray with multiplicity $r-1$.  Thus the complete root boundary
factor is the geometric sum
\begin{equation}\label{eq:Kroot}
 K(u)=u(r-1)+\frac{ur^2V(u)}{1-(r-1)V(u)}.
\end{equation}

\begin{theorem}[Exact root local zeta]\label{thm:algebraiclocal}
For the observation state $E=\{a_0\}$, the full first-return series and local
zeta function are
\begin{equation}\label{eq:exactFZ}
 \cF_{a_0}(u)=K(u)x(u),
 \qquad
 Z_{a_0}(u)=\frac1{1-K(u)x(u)},
\end{equation}
with $x$ and $K$ given by \eqref{eq:xclosed} and \eqref{eq:Kroot}.  In
particular $Z_{a_0}$ is algebraic of degree at most two over
$\mathbb Q(q,u)$.
\end{theorem}

\begin{proof}
After the initial state $a_0$ reaches $v_1$, every first return must first pass
through the stationary tail from $A_1$ to $B_0$; this contributes $x$.  The
last root stage is exactly the mutually exclusive collection summed by $K$.
The Markov property gives the product $Kx$, and the one-state first-return
identity gives \eqref{eq:exactFZ}.
\end{proof}

The closed form permits a complete singularity calculation.  We include the
details because they convert what might appear to be a numerical question into
an exact statement uniform in $q$.

\begin{theorem}[Exact dominant local poles]
\label{thm:exact-dominant-pole}
For every even $q\ge2$, the Taylor series of $Z_{a_0}(u)$ has radius
\begin{equation}\label{eq:exact-local-radius}
 R_u=\frac1{q+1}.
\end{equation}
Its only singularities on $|u|=R_u$ are simple poles at
$u=\pm(q+1)^{-1}$.
\end{theorem}

\begin{proof}
Put $t=u^2$.  Substitution of
$V=qt/(1-(q+1)t)$ into \eqref{eq:scatter1}--\eqref{eq:scatter2} gives
\begin{align}
 \Delta(u)&=\frac{1-(2q+1)t}{1-(q+1)t},
 \label{eq:Delta-simplified}\\
 \Delta(u)^2-4\alpha(u)\delta(u)
 &=\frac{(1-t)^2(1-4qt)}{(1-(q+1)t)^2}.
 \label{eq:discriminant-factorization}
\end{align}
Using $r=q/2+1$, the root factor also simplifies to
\begin{equation}\label{eq:K-simplified}
 K(u)=\frac{qu\bigl(1+(q+1)t\bigr)}
 {2-(q^2+2q+2)t}.
\end{equation}
Thus the first positive branch point of $x$ and the first positive pole of
$K$ are respectively
\[
 u_{\rm br}=\frac1{2\sqrt q},
 \qquad
 u_K=\sqrt{\frac2{q^2+2q+2}}.
\]
Both are strictly larger than $u_0=(q+1)^{-1}$.  The apparent pole of $V$ is
also farther away, and it cancels from the simplified expression for $x$.

At $u=u_0$ one has $V(u_0)=u_0$.  Direct substitution into
\eqref{eq:xclosed} and \eqref{eq:K-simplified} yields
\begin{equation}\label{eq:critical-Kx}
 x(u_0)=1,
 \qquad K(u_0)=1.
\end{equation}
Hence $1-K(u)x(u)$ vanishes at $u_0$.  The series $Kx=\cF_{a_0}$ has
nonnegative coefficients and a positive $u^2$ coefficient by \eqref{eq:f2};
it is therefore strictly increasing on $(0,u_0]$.  It follows that
$K(u)x(u)<1$ for $0\le u<u_0$.  By Pringsheim's theorem, a power series with
nonnegative coefficients and finite radius has a singularity at its positive
real radius.  Since the algebraic expression is analytic for $0\le u<u_0$,
this rules out a singularity of smaller modulus for
$Z_{a_0}=\sum_{m\ge0}(Kx)^m$.

If $|u|=u_0$ and $K(u)x(u)=1$, equality must hold in
$|K(u)x(u)|\le K(u_0)x(u_0)=1$.  Since the coefficients of both $u^2$ and
$u^4$ are positive by \eqref{eq:f2}--\eqref{eq:f4}, equality in the triangle
inequality forces $u^2/u_0^2=1$.  Hence $u=\pm u_0$, so there are no other
singularities on the convergence circle.  The derivative of $Kx$ at $u_0$ is
finite and strictly positive, so the zero of $1-Kx$ there is simple.  Evenness
from \cref{rem:periodtwo} gives the second simple pole at $-u_0$.
\end{proof}

The algebraic equation can be displayed without a square root.  If
$Z=Z_{a_0}$ and $F=1-Z^{-1}$, then \eqref{eq:xquadratic} with $F=Kx$ gives
\begin{equation}\label{eq:Zquadratic}
 \delta(Z-1)^2-\Delta KZ(Z-1)+\alpha K^2Z^2=0.
\end{equation}
All coefficients in \eqref{eq:Zquadratic} are rational functions of $q$ and
$u$.  Expansion of this equation reproduces \eqref{eq:f2}--\eqref{eq:f6} and
\eqref{eq:localexpansion}.  One additional exact numerical check is recorded
in \cref{tab:higher-coefficient-check}.
\begin{table}[ht]
\centering
\caption{A higher-order check of the algebraic root formula.}
\label{tab:higher-coefficient-check}
\begin{tabular}{c|rr}
\toprule
$q$ & $f_6(q)$ & $f_8(q)$ from \eqref{eq:exactFZ}\\
\midrule
$2$ & $120$ & $880$\\
$4$ & $2988$ & $48252$\\
$6$ & $21360$ & $606360$\\
\bottomrule
\end{tabular}
\end{table}

For the bulk renormalization we also need the homogeneous cell anchored at
$v_1$.  Cut off the exceptional root side, require minimal top coordinate one,
and observe that a cycle based on $a_1$ must return from $B_1$ by one vertical
excursion before taking $a_1$ again.  Both boundary indices are now one, so the
boundary factor is $uV$.

\begin{corollary}[Anchored cell zeta]\label{cor:cellzeta}
The Euler product over primitive homogeneous-tail cycles with minimal top
coordinate one is
\begin{equation}\label{eq:cellzeta}
 Z_{\rm cell}(u)=\frac1{1-uV(u)x(u)}.
\end{equation}
It is algebraic, and
\begin{equation}\label{eq:cellfirst}
 \log Z_{\rm cell}(u)=q(q-1)u^4+O(u^6).
\end{equation}
\end{corollary}

Here and below, the \emph{cut homogeneous tail} is the edge-state system
formed from the uniform cells with top coordinate $n\ge1$, after deleting the
exceptional root cell and every transition across its left boundary.  We use
the recurrent parity component containing $a_1$.  Its one-step based-loop
series at $a_1$ is $Z_{\rm cell}(u)$; after squaring the shift, the corresponding
variable is $v=u^2$.

\subsection{Exact pressure at infinity}
\label{sec:exact-pressure-infinity}

The anchored cell determinant also measures the exponential complexity seen
after every fixed finite core has been discarded.  This is the point at which
the branch singularity of the stationary tail acquires a dynamical meaning.
Put
\[
 u_{\mathrm{tail}}=\frac1{2\sqrt q},
 \qquad v_{\mathrm{tail}}=u_{\mathrm{tail}}^2=\frac1{4q},
 \qquad
 \widehat Z_{\rm cell}(v)=Z_{\rm cell}(\sqrt v).
\]
This is independent of the square-root branch because $Z_{\rm cell}$ is even.

\begin{proposition}[Pressure of the homogeneous tail]
\label{prop:homogeneous-tail-pressure}
The Taylor series of $Z_{\rm cell}(u)$ has exact radius
$u_{\mathrm{tail}}$.  Equivalently, the based-loop series
$\widehat Z_{\rm cell}(v)$ for the squared homogeneous tail has radius
$v_{\mathrm{tail}}$.  Hence the Gurevich pressure of the two-step
multiplicity potential on the cut homogeneous tail is
\begin{equation}\label{eq:tail-pressure}
 P_{\rm tail}=\log(4q).
\end{equation}
\end{proposition}

\begin{proof}
By the one-state local resolvent identity,
$Z_{\rm cell}(u)$ is the based-loop Green series of the cut homogeneous tail.
The discriminant factorization
\eqref{eq:discriminant-factorization} shows that the first possible
singularity of $x$ on the positive axis is
$u=u_{\mathrm{tail}}$.  The pole of $V$ lies farther away.  At this point,
direct substitution gives
\begin{equation}\label{eq:tail-boundary-values}
 V(u_{\mathrm{tail}})=\frac q{3q-1},
 \qquad
 x(u_{\mathrm{tail}})=\frac{2q-1}{\sqrt q},
 \qquad
 u_{\mathrm{tail}}V(u_{\mathrm{tail}})x(u_{\mathrm{tail}})
 =\frac{2q-1}{2(3q-1)}<1.
\end{equation}
Thus the denominator $1-uVx$ in \eqref{eq:cellzeta} does not vanish at the
first branch point.  The square-root term in \eqref{eq:xclosed} has nonzero
coefficient there, so the branch point is genuine and is inherited by
$Z_{\rm cell}$.

For $0\le u<u_{\mathrm{tail}}$, the series $uV(u)x(u)$ has nonnegative
coefficients and is strictly smaller than its value at
$u_{\mathrm{tail}}$, which is less than one by
\eqref{eq:tail-boundary-values}.  Hence $1-uVx$ has no zero in that interval.
For complex $u$ with $|u|<u_{\mathrm{tail}}$, the same conclusion follows
from
$|uV(u)x(u)|\le |u|V(|u|)x(|u|)<1$.  Therefore no singularity occurs in the
open disk, and the exact radius is $u_{\mathrm{tail}}$.  Passing to
$v=u^2$ gives radius $R=(4q)^{-1}$.  For an irreducible based-loop series the
Cauchy--Hadamard formula identifies its Gurevich pressure with $-\log R$;
hence \eqref{eq:tail-pressure} follows.  Translation
conjugates the two parity copies of the squared homogeneous tail, so this
pressure is independent of the parity copy used at the root.
\end{proof}

We next connect this cut-tail pressure with the pressure at infinity from
\eqref{eq:pressure-at-infinity-definition}.  For $N\ge1$, let $K_N$ consist
of those states in $\cC_{a_0}$ represented by the horizontal edges
$a_j,\bar a_j$ and cusp-mouth edges $b_{j,0},\bar b_{j,0}$ with $j<N$.
Removing $K_N$ leaves one recurrent component: the homogeneous tail beginning
at the $N$th cell.  The remaining components lie strictly inside severed
vertical rays and contain no positive periodic path.  All components and
boundaries in this argument are taken inside the communicating class
$\cC_{a_0}$; the other parity component of the full squared edge shift is not
part of the shift under consideration.
Here and below,
$\Sigma^{(2),+}_{a_0}\setminus K_N$ denotes the one-sided Markov shift
obtained by deleting the states in $K_N$ and retaining only transitions of
$B^{(2)}_{a_0}$ between the remaining states.

\begin{lemma}[Finite-core escape]
\label{lem:finite-core-escape}
For the locally finite squared root component and the bounded locally constant
potential $\phi_q$,
\begin{equation}\label{eq:finite-core-pressure-limit}
 P_\infty(\phi_q)
 =\lim_{N\to\infty}
 P_G\!\left(\phi_q\bigm|
 \Sigma^{(2),+}_{a_0}\setminus K_N\right).
\end{equation}
On the right, the pressure of a nontransitive restriction is the supremum of
the pressures of its recurrent components.
\end{lemma}

\begin{proof}
We give the excursion estimate specialized to the present locally finite
graph.  Let $P_N$ denote the pressure of the unique recurrent component of
the complement of $K_N$.  For every $\delta>0$, the total weight of length
$m$ paths between any two states in the finite boundary of that component is
at most
\begin{equation}\label{eq:open-tail-bound}
 C_\delta e^{(P_N+\delta)m}.
\end{equation}
To see this uniformly, for each ordered pair $(s,t)$ of boundary states choose
one positive path $\gamma_{t,s}$ from $t$ back to $s$.  Concatenation
$\pi\mapsto\pi\gamma_{t,s}$ injects the length-$m$ paths from $s$ to $t$ into
the based periodic paths at $s$.  The added lengths and the reciprocals of the
added weights range over finite sets because the boundary is finite.  The
definition of $P_N$ therefore bounds all sufficiently long nonzero periodic
partition sums by $e^{(P_N+\delta)m}$ up to one uniform multiplicative
constant; enlarging that constant handles the finitely many short lengths.
A severed vertical-ray excursion has two-step growth at most $q+1<4q$, by the
explicit weights in \eqref{eq:excursionweight}, so after one further increase
of $C_\delta$ the same estimate applies to every component of the complement.

A periodic path counted in \eqref{eq:escaping-periodic-pressure} decomposes at
its visits to $K_N$ into at most $s\le n/M+1$ outside excursions and joining
pieces in $K_N$.  If $k\le n/M$ is the total time spent in $K_N$, finiteness
of $K_N$ and its boundary bounds the total weight of all joining pieces with
specified endpoints by $D^{k+s}$ for a constant $D$.  The factors
$C_\delta^sD^{k+s}$, together with the choices of excursion lengths and visit
positions, contribute, for a constant $C_1>0$ independent of $n$ and $M$, at most
\[
 D^{n/M}\sum_{s\le n/M+1}\binom ns C_1^s
 =\exp\!\bigl(n\,o_M(1)\bigr),
 \qquad o_M(1)=O\!\left(\frac{1+\log M}{M}\right)
 \longrightarrow0\quad(M\to\infty).
\]
Combining this with \eqref{eq:open-tail-bound} gives
\[
 P_\infty(\phi_q;a_0,K_N,M)
 \le P_N+\delta+o_M(1).
\]
First let $M\to\infty$ and then $\delta\downarrow0$.  This proves the upper
bound in \eqref{eq:finite-core-pressure-limit}.

Conversely, fix a finite state set $F$.  Choose a homogeneous tail state $e$
beyond every horizontal cell meeting $F$, in the parity component of $a_0$.
There are fixed positive paths from $a_0$ to $e$ and from $e$ back to $a_0$;
write $L$ and $c>0$ for their total length and weight.  Appending these paths
to every length-$m$ loop based at $e$ and contained in the cut homogeneous
tail component produces distinct root-based loops of length $m+L$ and total
weight at least $c$ times the cut-tail partition sum.  The middle loop avoids
$F$, so the resulting root-based loop visits $F$ at most $L$ times.  For each
fixed $M$ it therefore satisfies the escape condition once $m\ge ML$.
Taking the limsup along lengths realizing the cut-tail pressure proves
$P_\infty(\phi_q)\ge P_N$.  Translation conjugates the two parity copies of
the homogeneous tail, so $P_N$ is independent of $N$ and the asserted limit
follows.
\end{proof}

\begin{theorem}[Exact pressure gap at infinity]
\label{thm:exact-pressure-gap}
For the two-step multiplicity potential $\phi_q$,
\begin{equation}\label{eq:exact-pressure-gap}
 P_\infty(\phi_q)=\log(4q)
 <2\log(q+1)=P_G(\phi_q).
\end{equation}
Consequently $\phi_q$ is strongly positive recurrent.
\end{theorem}

\begin{proof}
The pressure at infinity is $\log(4q)$ by
\cref{lem:finite-core-escape,prop:homogeneous-tail-pressure}.  On the other
hand, \cref{thm:exact-dominant-pole} shows that the based-loop series for the
squared root component has radius $(q+1)^{-2}$, and hence
$P_G(\phi_q)=2\log(q+1)$.  The strict inequality is equivalent to
$4q<(q+1)^2$, or $(q-1)^2>0$.  Since $q\ge2$, the pressure gap is strict.
The equivalence between this gap and strong positive recurrence applies
because $\phi_q$ is bounded, locally constant and has finite pressure
\cite{RuhrSarig2022,Velozo2026}.
\end{proof}

\subsection{A finite-boundary extension}
\label{sec:stationary-general}

The same argument applies when Schur elimination leaves finite boundary data,
even if a cusp contains countably many states:
\[
 \begin{aligned}
 \text{internal cell states}
 &\longrightarrow \text{boundary scattering}\\
 &\longrightarrow \text{stationary first passages}
 \longrightarrow \text{core return matrix }\cF_E.
 \end{aligned}
\]

\begin{definition}\label{def:stationarytail}
An edge-state system has an \emph{eventually stationary tail with finite
boundary data and algebraic scattering} if it is the union of a finite core,
finitely many exceptional internal components, and cells $\mathfrak C_n$,
$n\ge N$, for some $N\in\mathbb N_0$, arranged along a ray, with the following
properties.  The number of boundary channels is a fixed integer $d\ge1$.
\begin{enumerate}[label=(\roman*)]
\item Every interface is finite.  Each stationary cell has $d$ left-arrival
and $d$ right-arrival boundary states.
\item Translation identifies $\mathfrak C_n$ with $\mathfrak C_{n+1}$, including its boundary
states, internal transition system and all weights.
\item Intercell transitions are nearest-neighbour: leaving a cell through its
left interface enters the adjacent cell on the left, leaving through its right
interface enters the adjacent cell on the right, and no transition bypasses
an intervening interface or skips a cell.  Consequently every first passage
from a cell to its left neighbour either exits immediately to the left or
first exits to the right and returns through the translated half-tail.
\item A cell may have finitely or countably many internal states.  Its formal
internal resolvent is required to be coefficientwise well defined between
boundary states: if $T_{\rm int}$ is the internal transition block and
$\mathsf E_{\rm in},\mathsf E_{\rm out}$ are the boundary coupling blocks,
every coefficient of
$u^2\mathsf E_{\rm in}(I-uT_{\rm int})^{-1}\mathsf E_{\rm out}$ must be
finite.  After these internal states are eliminated, the four boundary
scattering blocks satisfy
\[
 \mathsf A(u),\mathsf B(u),\mathsf C(u),\mathsf D(u)
 \in uM_d(\mathbb K\langle u\rangle).
\]
\item Each exceptional internal component has a finite interface and, after
the same coefficientwise Schur elimination, has algebraic effective
scattering entries in $\mathbb K\langle u\rangle$.
\end{enumerate}
Here $\mathbb K$ is any characteristic-zero coefficient field containing the
transition weights,
\[
 \mathbb K\langle u\rangle
 =\{f(u)\in\mathbb K[[u]]:f(u)\text{ is algebraic over }\mathbb K(u)\}
\]
is the ring of algebraic formal power series, and $M_d(R)$ denotes the
$d\times d$ matrices over a ring $R$.  The prefactor $u$ says that every
boundary passage has positive length.
\end{definition}

Label the $d$ channels arriving at the $n$th cell from the left by
$L_n^1,\ldots,L_n^d$, and those arriving from the right by
$R_n^1,\ldots,R_n^d$.  Write $L_n$ and $R_n$ for the corresponding
$d$-dimensional channel spaces.  With rows representing initial channels and
columns representing final channels, the four effective blocks have the
types
\[
 \mathsf A:L_n\longrightarrow R_{n-1},\qquad
 \mathsf B:L_n\longrightarrow L_{n+1},\qquad
 \mathsf C:R_n\longrightarrow R_{n-1},\qquad
 \mathsf D:R_n\longrightarrow L_{n+1}.
\]
Thus $\mathsf A,\mathsf B$ give the left and right exits from $L_n$, while
$\mathsf C,\mathsf D$ give those from $R_n$.  A finite internal subsystem
gives rational blocks; the vertical cusp of the one-sided comb gives the
rational resolvent $V(u)$.

\paragraph{The matrix first-passage equations.}
We use the same row-to-column convention as for the transition matrix $B$:
the row records the initial channel, the column records the final channel,
and matrix products list successive path segments from left to right.  Define
$X(u),Y(u)\in M_d(\mathbb K[[u]])$ entrywise by
\begin{align*}
 X_{ij}(u)
 &=\sum_{\substack{\pi:L_n^i\leadsto R_{n-1}^j\\
                    \pi\text{ first reaches }R_{n-1}\text{ at its end}}}
   W(\pi)u^{\ell(\pi)},\\
 Y_{ij}(u)
 &=\sum_{\substack{\pi:R_n^i\leadsto R_{n-1}^j\\
                    \pi\text{ first reaches }R_{n-1}\text{ at its end}}}
   W(\pi)u^{\ell(\pi)}.
\end{align*}
Here $W(\pi)$ is the product of transition multiplicities along $\pi$.
Stationarity makes these matrices independent of $n$.  A first passage exits
immediately to the left or first exits right and returns through the translated
half-tail.  The nearest-neighbour condition in
\cref{def:stationarytail} makes this decomposition exhaustive and unique.
In chronological order it gives
\begin{equation}\label{eq:matrix-riccati}
 X=\mathsf A+\mathsf BXY,
 \qquad
 Y=\mathsf C+\mathsf DXY.
\end{equation}
For the one-sided comb, $d=1$, $L_n=\operatorname{span}\{A_n\}$,
$R_n=\operatorname{span}\{B_n\}$ and
$(\mathsf A,\mathsf B,\mathsf C,\mathsf D)
=(\alpha,\beta,\gamma,\delta)$; thus \eqref{eq:matrix-riccati} reduces
exactly to the scalar system \eqref{eq:xy-system}.

\Needspace{11\baselineskip}
\begin{theorem}[Algebraicity of stationary-tail local determinants]
\label{thm:stationary-algebraicity}
Let an edge-state system have an eventually stationary tail with finite
boundary data and algebraic scattering, and let $E$ be a finite observation
set in the core.  Then every entry of its first-return matrix $\cF_E(u)$ is
algebraic over $\mathbb K(u)$.  Consequently
\[
 Z_E(u)=\det(I_E-\cF_E(u))^{-1}
\]
is algebraic over $\mathbb K(u)$.  The same conclusion holds for an eventually
$p$-periodic tail, where $p\ge1$.
\end{theorem}

\begin{proof}
Because all four scattering blocks have zero constant term,
\eqref{eq:matrix-riccati} determines a unique solution
$(X,Y)\in uM_d(\mathbb K[[u]])^2$: after the coefficients below degree $m$
are known, comparison of coefficients determines the degree-$m$ terms.  This
gives a direct formal construction of the branch through $(0,0)$.

For algebraicity, set
\[
 \Phi(u,X,Y)=
 \bigl(X-\mathsf A-\mathsf BXY,\,
       Y-\mathsf C-\mathsf DXY\bigr).
\]
This is a finite polynomial system in the $2d^2$ entries of $(X,Y)$ with
coefficients in $\mathbb K\langle u\rangle$.  At
$(u,X,Y)=(0,0,0)$ its Jacobian with respect to $(X,Y)$ is the identity.
The algebraic power-series ring is Henselian; equivalently, the algebraic
implicit-function theorem applies (see, for example, the review and
references in \cite{Rond2018}).  It follows that the unique branch has
$X,Y\in M_d(\mathbb K\langle u\rangle)$, so every entry is algebraic over
$\mathbb K(u)$.

Attach the finite core and the finitely many exceptional components.  Their
effective scattering entries are rational or algebraic by definition.
Combining them with $X$ and $Y$ uses only finitely many additions, products
and Schur complements.  Every matrix that is inverted has identity constant
term, so its formal inverse exists and still has algebraic entries.  Hence
every entry of $\cF_E$ is algebraic.  Since
$\det(I_E-\cF_E(0))=1$, its reciprocal determinant is algebraic as well.  If the
tail is $p$-periodic, group $p$ consecutive cells into one supercell; finite
composition preserves both nearest-neighbour gluing and algebraic boundary
scattering, so the stationary argument applies to that supercell.
\end{proof}

\begin{remark}
The theorem is local: it does not claim that an arbitrary infinitely cusped
graph has a stationary algebraic-scattering tail, nor that its ordinary
global zeta exists.  It applies precisely when infinite cell interiors can be
compressed to finite algebraic boundary data.
\end{remark}

\section{Height regularization and the renormalized finite part}
\label{sec:regularized}

Here $Z(u,\rho)$ denotes the height-regularized family and $Z_{\rm ren}(u)$
its finite part after subtraction of the divergence as $\rho\uparrow1$.  The
construction proceeds in the order
\[
 T_\rho\ \longrightarrow\
 \text{root/cell orbit split}\ \longrightarrow\
 \frac{\mathscr P(u)}{1-\rho}+\text{finite part}\ \longrightarrow\
 Z_{\rm ren}(u).
\]
The last subsection compares this Abel procedure with a sharp height cutoff.

\subsection{Height damping and the Fredholm determinant}

Let
\[
 h(e)=d_{Y_q}(o(e),v_0).
\]
For a closed edge path $C=(e_0,\ldots,e_{m-1})$, define its integrated height
\begin{equation}\label{eq:integratedheight}
 H(C)=\sum_{j=0}^{m-1}h(e_j).
\end{equation}
This is invariant under cyclic rotation.  Let $M_\rho$ be the diagonal operator
on $\ell^2(\OrE(Y_q))$ given by
\[
 M_\rho\delta_e=\rho^{h(e)}\delta_e,
 \qquad0<\rho<1,
\]
where $(\delta_e)_{e\in\OrE(Y_q)}$ is the standard orthonormal basis, and set
\begin{equation}\label{eq:Trho}
 T_\rho=BM_\rho.
\end{equation}
Thus a closed matrix product around $C$ acquires the factor $\rho^{H(C)}$.

\begin{theorem}[Fredholm determinant]\label{thm:fredholm}
For $0<\rho<1$, the operator $T_\rho=BM_\rho$ is trace class, and
$Z(u,\rho):=\det(I-uT_\rho)^{-1}$ is meromorphic in $u\in\mathbb C$ with
entire reciprocal.  For $|u|$ sufficiently small,
\begin{align}
 \log Z(u,\rho)
 &=\sum_{m\ge1}\frac{u^m}{m}\Tr(T_\rho^m),\label{eq:traceexpansion}\\
 Z(u,\rho)
 &=\prod_{[C]\ \mathrm{primitive}}
 \left(1-W(C)\rho^{H(C)}u^{\ell(C)}\right)^{-1}.
\label{eq:regularizedeuler}
\end{align}
\end{theorem}

\begin{proof}
At height $m\ge1$, there is one top vertex of quotient degree $3$ and $m$
vertical vertices of quotient degree $2$.  Hence there are $2m+3$ oriented
edge states with origin at height $m$, while the root has two.  Therefore
\begin{equation}\label{eq:Mtrace}
 \|M_\rho\|_1
 =\sum_{e\in\OrE(Y_q)}\rho^{h(e)}
 =2+\sum_{m\ge1}(2m+3)\rho^m
 =2+\frac{3\rho}{1-\rho}+\frac{2\rho}{(1-\rho)^2}<\infty,
\end{equation}
where $\|\,\cdot\,\|_1$ is the trace norm.  Thus $M_\rho$ is trace class.
Every row sum of $B$ is $q+1$, while each column has at most three nonzero
entries, all bounded by $q+1$.  The Schur test gives
\[
 \|B\|\le\sqrt3\,(q+1).
\]
Hence $B$ is bounded, and the trace-class ideal property gives
\[
 \|T_\rho\|_1\le\|B\|\,\|M_\rho\|_1<\infty.
\]
The function $\det(I-uT_\rho)$ is therefore entire in $u$.  The standard trace
expansion for a trace-class determinant gives \eqref{eq:traceexpansion} near
the origin.  Expanding the trace in the oriented-edge basis and grouping
closed paths into powers of primitive cycles gives
\eqref{eq:regularizedeuler}.  See \cite{Simon2005,Deitmar2015} for the
corresponding determinant principles.
\end{proof}

\begin{remark}
The placement $T_\rho=BM_\rho$ is not arbitrary.  It assigns
$\rho^{h(e_{j+1})}$ to each transition $e_j\to e_{j+1}$, so a closed product
has the cyclically invariant weight $\rho^{H(C)}$.  Symmetric damping
$M_\rho^{1/2}BM_\rho^{1/2}$ gives the same closed-cycle weights and the same
nonzero Fredholm spectrum.
\end{remark}
\subsection{The first divergent coefficient}

For the cycle $C_n$ in \cref{cor:combdiv},
\[
 h(a_n)=n,\quad h(\bar a_n)=n+1,
 \quad h(b_{n,0})=n,\quad h(\bar b_{n,0})=n+1.
\]
Thus
\begin{equation}\label{eq:Cnheight}
 H(C_n)=4n+2.
\end{equation}

\begin{proposition}[Length-four cusp contribution]\label{prop:D4}
The cycles $C_n$, $n\ge1$, contribute
\begin{equation}\label{eq:D4}
 D_4(\rho)
 =\sum_{n\ge1}W(C_n)\rho^{H(C_n)}
 =\frac{q(q-1)\rho^6}{1-\rho^4}
\end{equation}
to the coefficient of $u^4$ in $\log Z(u,\rho)$.  As $\rho\uparrow1$,
\begin{equation}\label{eq:D4asymp}
 D_4(\rho)=\frac{q(q-1)}{4(1-\rho)}+O(1).
\end{equation}
\end{proposition}

\begin{proof}
Combine \eqref{eq:Cnweight} and \eqref{eq:Cnheight} and sum the geometric
series.  In the trace formula each primitive cycle has four pointed rotations,
which cancel the factor $1/4$ in \eqref{eq:traceexpansion}.
\end{proof}

This calculation is the first coefficient of the complete counterterm
constructed below.

\subsection{Primitive-cell factorization}

Horizontal translation creates infinitely many copies of the same tail cycle.
To choose one representative, translate the cycle until its smallest top
coordinate is one; we call that representative \emph{anchored}.  This removes
the translation multiplicity without changing its length or weight.

Let $\mathcal P_{\rm cell}$ be the collection of primitive positive cycles in
the homogeneous part of the one-sided comb tail whose minimal top coordinate
is one.  Every
such cycle uses $a_1$; otherwise it could neither leave nor close at its minimal
top vertex.  Consequently its Euler product is exactly
$Z_{\rm cell}$ from \cref{cor:cellzeta}.  For
$P\in\mathcal P_{\rm cell}$ put
\begin{equation}\label{eq:cell-data}
 \ell_P=\ell(P),\qquad w_P=W(P),\qquad \eta_P=H(P),
\end{equation}
where $P$ is anchored at minimal top coordinate one.

Every primitive cycle not meeting $v_0$ has a unique minimal top coordinate
$n\ge1$.  It is the translate of a unique $P\in\mathcal P_{\rm cell}$; denote
this translated cycle by $P_n$.  Its length and weight are unchanged, while
its integrated height becomes
\begin{equation}\label{eq:translatedheight}
 H(P_n)=\eta_P+(n-1)\ell_P.
\end{equation}
Let $Z_{\rm root}(u,\rho)$ denote the regularized Euler product over primitive
cycles meeting $v_0$.  We write $Z_{\rm root}(u)$ for the same root-meeting
Euler product without height damping.  It is well defined near $u=0$ by the
local theory, since all of these cycles meet a fixed finite edge-state set.
In general $Z_{\rm root}\ne Z_{a_0}$: the former also includes cycles confined
to the vertical ray attached to $v_0$.
The preceding partition gives the exact factorization
\begin{equation}\label{eq:primitivefactorization}
 Z(u,\rho)=Z_{\rm root}(u,\rho)
 \prod_{P\in\mathcal P_{\rm cell}}\prod_{n\ge1}
 \left(1-w_Pu^{\ell_P}
 \rho^{\eta_P+(n-1)\ell_P}\right)^{-1}.
\end{equation}
For fixed $P$, the inner product is
\[
 (w_Pu^{\ell_P}\rho^{\eta_P};\rho^{\ell_P})_\infty^{-1},
 \qquad
 (a;s)_\infty:=\prod_{j=0}^\infty(1-as^j),
\]
the reciprocal infinite shifted factorial.  Taking logarithms gives
\begin{equation}\label{eq:bulklog}
 \log\frac{Z(u,\rho)}{Z_{\rm root}(u,\rho)}
 =\sum_{P\in\mathcal P_{\rm cell}}\sum_{k\ge1}
 \frac{(w_Pu^{\ell_P})^k\rho^{k\eta_P}}
 {k(1-\rho^{k\ell_P})}.
\end{equation}

\subsection{The complete counterterm and uniform Abel finite part}

For one orbit and its $k$th repetition, \eqref{eq:bulklog} gives
\[
 \frac1{k(1-\rho^{k\ell_P})}
 \sim \frac1{k^2\ell_P(1-\rho)}
 \qquad(\rho\uparrow1).
\]
Summation over $k$ gives $\ell_P^{-1}\Li(w_Pu^{\ell_P})$, motivating the
counterterm below.

Define the bulk \emph{integrated-pressure counterterm}
\begin{equation}\label{eq:pressure-counterterm}
 \mathscr P(u)=\sum_{P\in\mathcal P_{\rm cell}}
 \frac1{\ell_P}\Li(w_Pu^{\ell_P}).
\end{equation}
Because
$Z_{\rm cell}(u)=\prod_P(1-w_Pu^{\ell_P})^{-1}$, differentiation term by
term gives the useful algebraic integral representation
\begin{equation}\label{eq:pressure-integral}
 \mathscr P(u)=\int_0^u\frac{\log Z_{\rm cell}(t)}t\,dt
 =-\int_0^u\frac{\log(1-tV(t)x(t))}t\,dt.
\end{equation}
Thus the complete divergent coefficient is explicitly computable from the
algebraic function \eqref{eq:cellzeta}, even though its primitive-orbit form
naturally involves dilogarithms.

For later use, define also the height moment series
\begin{equation}\label{eq:heightmoment}
 \mathscr H(u)=-\sum_{P\in\mathcal P_{\rm cell}}
 \frac{\eta_P}{\ell_P}\log(1-w_Pu^{\ell_P}).
\end{equation}

The next definition isolates the features of the comb that are actually used
in the renormalization.  It applies to a transition system whether or not the
individual stationary cells have finitely many internal states.

\begin{definition}[Stationary translated-orbit system]
\label{def:translated-orbit-system}
Let $\mathcal P$ be a collection of primitive cycles equipped with maps
\[
 \ell:\mathcal P\to\mathbb Z_{\ge1},\qquad
 W:\mathcal P\to\mathbb C,\qquad
 H:\mathcal P\to\mathbb N_0,
\]
recording length, weight and integrated height.  For $0<\rho<1$, the
height-regularized Euler factor of $C\in\mathcal P$ is
$(1-W(C)u^{\ell(C)}\rho^{H(C)})^{-1}$.  These primitive-orbit data have a
\emph{stationary translated-orbit decomposition with growth constant
$\Lambda>0$} if $\mathcal P$ splits into a root family
$\mathcal P_{\rm root}$ and families
\[
 \{P_n:n\ge1\},\qquad P\in\mathcal P_{\rm cell},
\]
Write $\ell_P=\ell(P_1)$, $w_P=W(P_1)$ and $\eta_P=H(P_1)$.  The following
conditions are required.
\begin{enumerate}[label=(\roman*)]
\item Translation preserves length and weight and shifts integrated height
affinely:
\[
 \ell(P_n)=\ell_P,\qquad W(P_n)=w_P,\qquad
 H(P_n)=\eta_P+(n-1)\ell_P.
\]
\item For some constants $C,C_0>0$,
\begin{equation}\label{eq:abstract-orbit-growth}
 \sum_{\substack{P\in\mathcal P_{\rm cell}\\\ell_P=m}}|w_P|
 \le C\Lambda^m,
 \qquad |w_P|\le\Lambda^{\ell_P},
 \qquad \eta_P\le C_0\ell_P^2.
\end{equation}
\item The products
\[
 Z_{\rm root}(u,\rho)
 =\prod_{C\in\mathcal P_{\rm root}}
   (1-W(C)u^{\ell(C)}\rho^{H(C)})^{-1},
 \qquad
 Z_{\rm root}(u)
 =\prod_{C\in\mathcal P_{\rm root}}
   (1-W(C)u^{\ell(C)})^{-1}
\]
converge normally on $|u|<\Lambda^{-1}$, and
$Z_{\rm root}(u,\rho)\to Z_{\rm root}(u)$ there locally uniformly as
$\rho\uparrow1$.
\end{enumerate}
The regularized Euler product associated with this decomposition is
\begin{equation}\label{eq:abstract-regularized-product}
 Z(u,\rho)=Z_{\rm root}(u,\rho)
 \prod_{P\in\mathcal P_{\rm cell}}\prod_{n\ge1}
 \left(1-w_Pu^{\ell_P}
 \rho^{\eta_P+(n-1)\ell_P}\right)^{-1}.
\end{equation}
Consequently its bulk logarithm has the form
\begin{equation}\label{eq:abstract-bulk-log}
 \sum_{P\in\mathcal P_{\rm cell}}\sum_{k\ge1}
 \frac{(w_Pu^{\ell_P})^k\rho^{k\eta_P}}
 {k(1-\rho^{k\ell_P})}.
\end{equation}
We write
\begin{equation}\label{eq:abstract-cell-zeta}
 Z_{\rm cell}(u)=
 \prod_{P\in\mathcal P_{\rm cell}}(1-w_Pu^{\ell_P})^{-1}
\end{equation}
for the anchored cell Euler product, initially as an analytic germ at the
origin.
\end{definition}

Before treating the whole orbit family, it is useful to keep one anchored
primitive cell cycle $P$ fixed.  Put $z_P=w_Pu^{\ell_P}$ and assume
$|z_P|<1$.  Then
\[
 \Phi_P(u,\rho)=\sum_{k\ge1}
 \frac{z_P^k\rho^{k\eta_P}}{k(1-\rho^{k\ell_P})}.
\]
Expanding the geometric denominator at $\rho=1$ and then summing over the
repetition number $k$ gives
\begin{align}\label{eq:one-orbit-warmup}
 \Phi_P(u,\rho)
 &=\frac{\Li(z_P)}{\ell_P(1-\rho)}
 +\left(\frac{\eta_P}{\ell_P}-\frac12\right)\log(1-z_P)
-\frac{\Li(z_P)}{2\ell_P}+o(1).
\end{align}
Here the error tends to zero as $\rho\uparrow1$, locally uniformly where
$|z_P|<1$.
The first term is the removed divergence and the other two form the finite
contribution of this orbit family.  Since the fixed-$P$ expansion does not
justify summation over all primitive cells, the following family estimate
provides the uniform interchange of limit and sum.

\begin{lemma}[Uniform Abelian finite part]\label{lem:uniform-abel}
Let $J$ be countable, let $\ell_j\ge1$ and $\eta_j\ge0$ be integers satisfying
$\eta_j\le C_0\ell_j^2$ for some $C_0>0$, and let $z_j$ be holomorphic on a
disk $D$.  Suppose
that for every compact $K\Subset D$,
\begin{equation}\label{eq:abel-summability}
 \sum_{j\in J}\sum_{k\ge1}
 k^2(1+\ell_j)^3\sup_{u\in K}|z_j(u)|^k<\infty.
\end{equation}
Here $K\Subset D$ means that $K$ is a compact subset of $D$.
For $0<\rho<1$ put
\[
 \Phi_j(u,\rho)=\sum_{k\ge1}
 \frac{z_j(u)^k\rho^{k\eta_j}}{k(1-\rho^{k\ell_j})}.
\]
Then, locally uniformly on $D$,
\begin{align}
 &\sum_{j\in J}\left\{
 \Phi_j(u,\rho)-
 \frac{\Li(z_j(u))}{\ell_j(1-\rho)}\right\}
 \longrightarrow \sum_{j\in J}\left\{
 \left(\frac{\eta_j}{\ell_j}-\frac12\right)
 \log(1-z_j(u))
 -\frac{\Li(z_j(u))}{2\ell_j}\right\}.
 \label{eq:abel-family-limit}
\end{align}
Both series on the right converge normally.
\end{lemma}

\begin{proof}
Set $\varepsilon=1-\rho$.  For a fixed triple $(k,\ell,\eta)$ define
\[
 R_{k,\ell,\eta}(\rho)=
 \frac{\rho^{k\eta}}{1-\rho^{k\ell}}
 -\frac1{k\ell\varepsilon}
 -\left(\frac12-\frac1{2k\ell}-\frac{\eta}{\ell}\right).
\]
Put $t=-\log\rho$,
$a=\eta/\ell$ and $x=k\ell t$.  Then
\begin{align}
 R_{k,\ell,\eta}(\rho)
 &=\left\{\frac{e^{-ax}}{1-e^{-x}}-\frac1x-
       \left(\frac12-a\right)\right\}\notag\\
 &\quad+\frac1{k\ell}
 \left(\frac1t-\frac1\varepsilon+\frac12\right).
 \label{eq:abel-remainder-split}
\end{align}
For $0<\varepsilon\le1/2$, the last parenthesis is $O(\varepsilon)$.
The elementary expansions of $e^{-ax}$ and
$(1-e^{-x})^{-1}$, with their integral remainders, give a constant depending
only on $C_0$ such that
\begin{align}
 |R_{k,\ell,\eta}(\rho)|
 &\le C\varepsilon k(1+\ell)^3,
 &&x\le1,\label{eq:abel-small-bound}\\
 |R_{k,\ell,\eta}(\rho)|
 &\le C(1+\ell),
 &&x>1.\label{eq:abel-large-bound}
\end{align}
Indeed $a\le C_0\ell$ and $t\le2\varepsilon$ in the first range; in the
second, $(1-e^{-x})^{-1}$ is uniformly bounded and $x^{-1}\le1$.

Multiply by $|z_j(u)|^k/k$ and sum.  On $x\le1$,
\eqref{eq:abel-summability} makes the total remainder
$O_K(\varepsilon)$.  On $x>1$, \eqref{eq:abel-large-bound} gives a summable
majorant independent of $\rho$, and the indicator of this range tends to zero
for each fixed $(j,k)$.  Dominated convergence therefore gives $o_K(1)$.
We have proved, uniformly for $u\in K$, the summable termwise expansion
\begin{equation}\label{eq:abel-expansion}
 \frac{\rho^{k\eta}}{1-\rho^{k\ell}}
 =\frac1{k\ell(1-\rho)}
 +\frac12-\frac1{2k\ell}-\frac{\eta}{\ell}+o(1).
\end{equation}
Finally use
$\sum_{k\ge1}z^k/k=-\log(1-z)$ and
$\sum_{k\ge1}z^k/k^2=\Li(z)$ to obtain
\eqref{eq:abel-family-limit}.  The same majorants prove normal convergence of
the limiting series.
\end{proof}

\begin{theorem}[Renormalization of a stationary translated tail]
\label{thm:abstract-renormalized-limit}
Let primitive-orbit data have a stationary translated-orbit decomposition with
growth constant $\Lambda$.  On $|u|<\Lambda^{-1}$ the
series
\begin{align}
 \mathscr P(u)&=\sum_{P\in\mathcal P_{\rm cell}}
 \frac1{\ell_P}\Li(w_Pu^{\ell_P}),\label{eq:abstract-counterterm}\\
 \mathscr H(u)&=-\sum_{P\in\mathcal P_{\rm cell}}
 \frac{\eta_P}{\ell_P}\log(1-w_Pu^{\ell_P})
 \label{eq:abstract-height-moment}
\end{align}
converge normally.  The limit
\begin{equation}\label{eq:renormalizedlimit}
 Z_{\rm ren}(u)=\lim_{\rho\uparrow1}
 \exp\!\left(-\frac{\mathscr P(u)}{1-\rho}\right)Z(u,\rho)
\end{equation}
exists locally uniformly on that disk and equals
\begin{align}
 Z_{\rm ren}(u)
 &=Z_{\rm root}(u)
 \prod_{P\in\mathcal P_{\rm cell}}
 (1-w_Pu^{\ell_P})^{\eta_P/\ell_P-1/2}
 \exp\!\left[-\frac{\Li(w_Pu^{\ell_P})}{2\ell_P}\right]
 \label{eq:renormalizedproduct}\\
 &=Z_{\rm root}(u)Z_{\rm cell}(u)^{1/2}
 \exp\!\left[-\mathscr H(u)-\frac12\mathscr P(u)\right].
 \label{eq:renormalizedcompact}
\end{align}
All logarithms and fractional powers are the analytic branches equal to zero
and one, respectively, at $u=0$.
\end{theorem}

\begin{proof}
Apply \cref{lem:uniform-abel} with
$J=\mathcal P_{\rm cell}$ and $z_P(u)=w_Pu^{\ell_P}$.  From
\eqref{eq:abstract-orbit-growth}, for every $k\ge1$,
\begin{equation}\label{eq:abstract-power-bound}
 \sum_{\ell_P=m}|w_P|^k
 \le\Lambda^{m(k-1)}
 \sum_{\ell_P=m}|w_P|
 \le C\Lambda^{mk}.
\end{equation}
If $K\Subset\{|u|<\Lambda^{-1}\}$, put
$r_K=\sup_{u\in K}|u|$ and $\theta=\Lambda r_K<1$.  Then
\begin{align*}
 &\sum_{P\in\mathcal P_{\rm cell}}\sum_{k\ge1}
 k^2(1+\ell_P)^3\sup_{u\in K}|w_Pu^{\ell_P}|^k\\
 &\qquad\le
 C\sum_{m\ge1}(1+m)^3\sum_{k\ge1}k^2\theta^{mk}<\infty.
\end{align*}
This is \eqref{eq:abel-summability}.  Since
$\eta_P/\ell_P\le C_0\ell_P$, the same majorant proves normal convergence of
$\mathscr P$ and $\mathscr H$ separately.

The logarithmic factorization \eqref{eq:abstract-bulk-log} and
\cref{lem:uniform-abel} now show that subtraction of
$\mathscr P(u)/(1-\rho)$ leaves, orbit by orbit, the normally convergent term
\begin{equation}\label{eq:finite-per-orbit}
 \left(\frac{\eta_P}{\ell_P}-\frac12\right)
 \log(1-w_Pu^{\ell_P})
 -\frac1{2\ell_P}\Li(w_Pu^{\ell_P}).
\end{equation}
Exponentiation and the root convergence in
\cref{def:translated-orbit-system} prove \eqref{eq:renormalizedproduct}.
Finally
\eqref{eq:renormalizedcompact} follows from the definitions of
$Z_{\rm cell}$, $\mathscr P$ and $\mathscr H$.
\end{proof}

\begin{corollary}[Complete renormalized limit for $Y_q$]
\label{thm:renormalizedlimit}
The one-sided comb system satisfies the hypotheses of
\cref{thm:abstract-renormalized-limit} with $\Lambda=q+1$.  Consequently
\eqref{eq:pressure-counterterm} and \eqref{eq:heightmoment} converge normally
on $|u|<(q+1)^{-1}$, and the locally uniform limit
\eqref{eq:renormalizedlimit} is given by
\eqref{eq:renormalizedproduct}--\eqref{eq:renormalizedcompact} throughout
that disk.
\end{corollary}

\begin{proof}
Let
\[
 \mathcal A_m=\sum_{\substack{P\in\mathcal P_{\rm cell}\\\ell_P=m}}w_P.
\]
Every anchored primitive cycle passes through $a_1$.  Pointing it at one of
its visits to $a_1$ injects its weight into the total weight of length-$m$
loops based at $a_1$.  Since every row sum of $B$ is $q+1$,
\begin{equation}\label{eq:anchored-path-bound}
 \mathcal A_m\le (B^m)_{a_1,a_1}\le(q+1)^m,
 \qquad
 w_P\le(q+1)^{\ell_P}.
\end{equation}
A length-$\ell_P$ anchored cycle cannot rise above height $1+\ell_P$;
consequently
\begin{equation}\label{eq:comb-height-bound}
 \eta_P\le\ell_P(1+\ell_P)\le2\ell_P^2.
\end{equation}
The translated-height identity is \eqref{eq:translatedheight}.  Root cycles
meet a fixed finite state set, and their total length-$m$ weight is bounded by
$C(q+1)^m$.  Since $0<\rho^{H(C)}\le1$ and it tends to one for each root
cycle, the Weierstrass test gives
$Z_{\rm root}(u,\rho)\to Z_{\rm root}(u)$ locally uniformly on
$|u|<(q+1)^{-1}$.  Thus all conditions in
\cref{def:translated-orbit-system} hold with $\Lambda=q+1$, and the result
follows from \cref{thm:abstract-renormalized-limit}.
\end{proof}

The root finite part is itself explicit.  Primitive root cycles meeting $a_0$
give $Z_{a_0}$.  Those not meeting $a_0$ are confined to the root vertical ray;
their first-return series based at $b_{0,0}$ is $(r-1)V$.  Hence
\begin{equation}\label{eq:rootfactor}
 Z_{\rm root}(u)=\frac{Z_{a_0}(u)}{1-(r-1)V(u)},
\end{equation}
which is algebraic by \cref{thm:algebraiclocal}.

The anchored cycle $C_1$ has $(\ell,w,\eta)=(4,q(q-1),6)$, so
\begin{equation}\label{eq:pressure-u4}
 [u^4]\mathscr P(u)=\frac{q(q-1)}4,
\end{equation}
in agreement with \eqref{eq:D4asymp}.

\subsection{Dependence on the height normalization}
\label{sec:height-covariance}

Local changes of height leave the finite part unchanged, whereas a global
shift changes it by an explicit bulk factor.

\begin{theorem}[Height covariance and local stability]
\label{thm:height-covariance}
Write $Z^{h}(u,\rho)$ for the height-regularized determinant and
$Z_{\rm ren}^{h}(u)$ for the corresponding renormalized finite part constructed
from the height $h$.
\begin{enumerate}[label=(\roman*)]
\item If $c\in\mathbb Z_{\ge0}$ and $h_c=h+c$, then on
$|u|<(q+1)^{-1}$,
\begin{equation}\label{eq:constant-height-covariance}
 Z_{\rm ren}^{h_c}(u)
 =Z_{\rm cell}(u)^{-c}Z_{\rm ren}^{h}(u).
\end{equation}
\item If $\widetilde h-h$ has finite support and $\widetilde h$ is a
nonnegative integer-valued height for which the damped operator is trace
class, then
\begin{equation}\label{eq:finite-height-stability}
 Z_{\rm ren}^{\widetilde h}(u)=Z_{\rm ren}^{h}(u)
 \qquad(|u|<(q+1)^{-1}).
\end{equation}
\end{enumerate}
\end{theorem}

\begin{proof}
For a closed path $C$, replacing $h$ by $h+c$ replaces $H(C)$ by
$H(C)+c\ell(C)$.  Hence
\begin{equation}\label{eq:shifted-regularized-zeta}
 Z^{h_c}(u,\rho)=Z^h(\rho^c u,\rho).
\end{equation}
The bulk counterterm itself is unchanged, because its leading Abel coefficient
does not involve $\eta_P$.  Using
\[
 \mathscr P(\rho^c u)
 =\mathscr P(u)-c(1-\rho)u\mathscr P'(u)+o(1-\rho)
\]
in the renormalized limit gives
\[
 Z_{\rm ren}^{h_c}(u)
 =\exp\!\bigl(-cu\mathscr P'(u)\bigr)Z_{\rm ren}^{h}(u).
\]
Differentiating \eqref{eq:pressure-integral} yields
$u\mathscr P'(u)=\log Z_{\rm cell}(u)$, proving
\eqref{eq:constant-height-covariance}.

Now suppose $\widetilde h-h$ has finite support.  The two integrated heights
agree on every primitive translated cycle sufficiently far from that support,
so their complete counterterms coincide.  In the quotient of the two
regularized Euler products, only primitive cycles meeting a fixed finite state
set remain.  Their total absolute weight in length $m$ is bounded by
$C(q+1)^m$.  Each quotient factor tends to one as $\rho\uparrow1$, and this
geometric bound gives locally uniform dominated convergence on
$|u|<(q+1)^{-1}$.  Thus the quotient of the two renormalized limits tends to
one, proving \eqref{eq:finite-height-stability}.
\end{proof}

\subsection{Finite exhaustions and sharp-cutoff limits}
\label{sec:exhaustion}

We now compare the Abel height regularization with a sharp finite-height
cutoff.  The comparison separates the universal bulk pressure from the finite
geometric data retained by the two limiting procedures.

Let $\Pi_N$ project onto oriented edges whose origin has height at most $N$, and
put
\[
 B_{\le N}=\Pi_NB\Pi_N,
 \qquad
 T_{\rho,\le N}=\Pi_NT_\rho\Pi_N.
\]
Both act on finite-dimensional spaces.
Since $T_\rho$ is trace class and $\Pi_N\to I$ strongly,
\[
 \|T_{\rho,\le N}-T_\rho\|_1\longrightarrow0,
 \qquad
 \det(I-uT_{\rho,\le N})\longrightarrow\det(I-uT_\rho)
\]
locally uniformly in $u\in\mathbb C$.  This is the standard trace-norm
approximation of a trace-class operator by its finite-rank coordinate
compressions, together with continuity of the Fredholm determinant.

Without height damping, the length-four cusp contribution in a height-$N$
truncation is
\begin{equation}\label{eq:linearN}
 \sum_{n=1}^{N-1}W(C_n)=(N-1)q(q-1).
\end{equation}
Thus $\log\det(I-uB_{\le N})^{-1}$ has a coefficient growing linearly in $N$.
Define the finite-volume pressure and relative determinant by
\begin{align}
 p(u)&=\lim_{N\to\infty}\frac1N\log\det(I-uB_{\le N})^{-1},
 \label{eq:pressurelimit}\\
 Z_{\rm rel}(u)&=\lim_{N\to\infty}
 \exp(-N p(u))\det(I-uB_{\le N})^{-1}.
 \label{eq:relativelimit}
\end{align}
Here $p(u)$ is a per-cell logarithmic-zeta series; it should not be confused
with the scalar Gurevich pressure $P_G$ introduced in
\cref{sec:weighted-thermo-definitions}.
For $P\in\mathcal P_{\rm cell}$ let
\begin{equation}\label{eq:maxheight}
 m_P=\max\{h(e):e\text{ occurs in the anchored cycle }P\}.
\end{equation}

\begin{theorem}[Finite-volume pressure and relative limit]
\label{thm:finite-renormalized}
There exists $\varepsilon_q>0$ such that, on $|u|<\varepsilon_q$, both limits
\eqref{eq:pressurelimit}--\eqref{eq:relativelimit} exist locally uniformly and
\begin{align}
 p(u)&=\log Z_{\rm cell}(u),\label{eq:finitepressureexact}\\
 Z_{\rm rel}(u)&=Z_{\rm root}(u)
 \prod_{P\in\mathcal P_{\rm cell}}
 (1-w_Pu^{\ell_P})^{m_P-1}.
 \label{eq:finiterelativeexact}
\end{align}
In particular $[u^4]p(u)=q(q-1)$, as forced by \eqref{eq:linearN}.
\end{theorem}

\begin{proof}
An anchored cycle $P$ translated to minimal top coordinate $n$ lies in the
height-$N$ ball precisely when $m_P+n-1\le N$.  Hence, for $N\ge m_P$, it has
exactly $N-m_P+1$ allowed translates.  The primitive-cycle factorization of
the finite determinant therefore has a root part $Z_{{\rm root},N}(u)$ and the
bulk exponents $N-m_P+1$.

A root-meeting cycle of length $m$ cannot reach height greater than $m$, so
the corresponding coefficients stabilize once $N\ge m$.  The root path bound
$C(q+1)^m$ used in the proof of \cref{thm:renormalizedlimit} then gives,
for sufficiently small $|u|$,
\[
 Z_{{\rm root},N}(u)\longrightarrow Z_{\rm root}(u)
\]
locally uniformly, independently of the bulk translation count.

For the bulk part, dividing its logarithm by $N$ gives the Euler product for
$\log Z_{\rm cell}$; multiplying by $Z_{\rm cell}^{-N}$ leaves the exponents
$m_P-1$.  Since $m_P\le1+\ell_P$, the exponential path bound used in
\cref{thm:renormalizedlimit} gives normal convergence for sufficiently small
$|u|$ and justifies passage to the limit.
\end{proof}

The two complete renormalizations retain different finite geometric data.
The sharp height cutoff uses the maximal height $m_P$, whereas Abel height
damping uses the integrated height $\eta_P$ and produces the dilogarithmic
counterterm \eqref{eq:pressure-counterterm}.  Their bulk terms are related by
\eqref{eq:pressure-integral}:
\[
 \mathscr P(u)=\int_0^u\frac{p(t)}t\,dt.
\]

\section{Thermodynamic consequences and concluding remarks}
\label{sec:thermo}

\subsection{Ruelle--Perron--Frobenius consequences}

Geodesic flow on a tree quotient admits a coding by a locally compact
countable-state Markov shift
\cite{BroisePaulinCoding2007,BroisePaulinModular2007,
BroiseParkkonenPaulin2019,BroiseParkkonenPaulin2021}.
The potential $\phi_q$ of \cref{def:multiplicity-potential} retains the number
of parallel symbolic branches in this coding.  By
\cref{def:multiplicity-potential,thm:exact-pressure-gap}, the squared root
component is mixing and $\phi_q$ is a bounded, locally constant strongly
positive recurrent potential.  It therefore lies in the setting of Sarig,
Cyr--Sarig and R\"uhr--Sarig
\cite{Sarig1999,Sarig2003,SarigPhase2001,CyrSarig2009,RuhrSarig2022}.

For reference, the associated Ruelle transfer operator is
\begin{equation}\label{eq:ruelle-operator}
 (\mathcal L_{\phi_q}g)(x)
 =\sum_{\sigma_2 y=x}e^{\phi_q(y)}g(y).
\end{equation}

\begin{corollary}[Equilibrium state and spectral gap]
\label{cor:rpf-consequences}
The potential $\phi_q$ admits a unique Ruelle--Perron--Frobenius equilibrium
probability measure $\mu_{\phi_q}$ on $\Sigma^{(2),+}_{a_0}$.  More precisely,
there are a positive eigenfunction $h_q$ and an eigenmeasure $\nu_q$, normalized
by $\int h_q\,d\nu_q=1$, such that
\[
 \mathcal L_{\phi_q}h_q=e^{P_G(\phi_q)}h_q,
 \qquad
 \mathcal L_{\phi_q}^{*}\nu_q=e^{P_G(\phi_q)}\nu_q,
 \qquad
 \mu_{\phi_q}=h_q\nu_q.
\]
The normalized transfer operator
\begin{equation}\label{eq:normalized-ruelle-operator}
 \widehat{\mathcal L}_{\phi_q}g
 =e^{-P_G(\phi_q)}h_q^{-1}
   \mathcal L_{\phi_q}(h_qg)
\end{equation}
has a spectral-gap realization on the Banach space constructed by Cyr--Sarig
\cite{CyrSarig2009}.  Consequently the exponential correlation bounds proved
there apply to the corresponding classes of observables.
\end{corollary}

\begin{proof}
The potential is weakly H\"older, has finite pressure and satisfies the exact
strong-positive-recurrence gap.  The Ruelle--Perron--Frobenius and spectral-gap
conclusions follow from
\cite{Sarig1999,SarigPhase2001,CyrSarig2009,RuhrSarig2022}; see also
\cite{Sarig2003} for Gibbs-measure existence.  Boundedness gives
$\int\phi_q\,d\mu_{\phi_q}>-\infty$, so this probability measure is an
equilibrium state, and uniqueness follows from \cite{BuzziSarig2003}.
\end{proof}

\subsection{Local recurrence and orbit counting}

For a finite observation set $E$ in the present bipartite model, introduce
$v=u^2$ and write
\[
 \widehat Z_E(v)=Z_E(\sqrt v).
\]
This is independent of the choice of square-root branch because $Z_E$ is even.
Define the weighted \emph{local orbit coefficients} $N_E(n)$ by
\begin{equation}\label{eq:NE-definition}
 v\frac{d}{dv}\log\widehat Z_E(v)
 =\sum_{n\ge1}N_E(n)v^n.
\end{equation}
If $\mathcal P_E$ is the set of primitive cycle classes meeting $E$ and
$n(C)=\ell(C)/2$, the local Euler product gives
\begin{equation}\label{eq:NE-orbit-meaning}
 N_E(n)=\sum_{d\mid n}d
 \sum_{\substack{[C]\in\mathcal P_E\\n(C)=d}}W(C)^{n/d}.
\end{equation}
Thus $N_E(n)$ is the multiplicity-weighted count of periodic points lying on
orbit classes that meet $E$; it is not the based count $(B^{2n})_{e,e}$ at a
specified present state.

\begin{theorem}[Local-orbit asymptotics]
\label{thm:local-orbit-counting}
For the squared root component $\Sigma^{(2)}_{a_0}$, put
\begin{equation}\label{eq:Rq-sigmaq}
 R_q=\frac1{(q+1)^2},
 \qquad
 \sigma_q=\min\left\{\frac1{4q},
 \frac2{q^2+2q+2}\right\}.
\end{equation}
Then $\sigma_q>R_q$, and there exists $R_{1,q}$ with
$R_q<R_{1,q}<\sigma_q$ such that
\begin{equation}\label{eq:local-pole-factorization}
 \widehat Z_{a_0}(v)
 =\frac{G_q(v)}{1-v/R_q}
 \qquad(|v|<R_{1,q}),
\end{equation}
where $G_q$ is holomorphic and nonzero.  Consequently, for every
$R_2\in(R_q,R_{1,q})$,
\begin{equation}\label{eq:periodic-point-asymptotic}
 N_{a_0}(n)=(q+1)^{2n}+O(R_2^{-n}).
\end{equation}
\end{theorem}

\begin{proof}
Let $\widehat{\cF}_{a_0}(v)=\cF_{a_0}(\sqrt v)$.  Equations
\eqref{eq:xclosed}, \eqref{eq:discriminant-factorization} and
\eqref{eq:K-simplified} have the form $x(u)=u\xi_q(u^2)$ and
$K(u)=u k_q(u^2)$.  The first possible branch singularity of $\xi_q(v)$ is
$v=(4q)^{-1}$, while the first pole of the rational function $k_q(v)$ is
$v=2/(q^2+2q+2)$.  Thus
$\widehat{\cF}_{a_0}(v)=v\xi_q(v)k_q(v)$ is holomorphic for
$|v|<\sigma_q$.  Both entries in the minimum defining $\sigma_q$ are larger
than $R_q$: the two inequalities reduce respectively to
$(q-1)^2>0$ and $q(q+2)>0$.

By \eqref{eq:critical-Kx} and \cref{thm:exact-dominant-pole},
$1-\widehat{\cF}_{a_0}$ has a simple zero at $R_q$ and no other zero on
$|v|=R_q$.  Since $\widehat{\cF}_{a_0}$ is holomorphic on
$|v|<\sigma_q$, one may choose
$R_{1,q}\in(R_q,\sigma_q)$ with no other zero in $|v|<R_{1,q}$.
Removing the simple factor $1-v/R_q$ gives the holomorphic nonvanishing
function $G_q$ in \eqref{eq:local-pole-factorization}.  Taking the logarithmic
derivative gives
\[
 v\frac{d}{dv}\log\widehat Z_{a_0}(v)
 =\sum_{n\ge1}R_q^{-n}v^n+v\frac{G_q'(v)}{G_q(v)}.
\]
Cauchy's estimate on $|v|=R_2<R_{1,q}$ proves
\eqref{eq:periodic-point-asymptotic}.
\end{proof}

\subsection{Outlook and conclusion}
\label{sec:questions}

Three directions remain: analytic continuation of the finite part, comparison
of genuinely different regularization schemes, and its relation to
group-theoretic orbit data.

The analytic germs $\mathscr P$, $\mathscr H$ and $Z_{\rm ren}$ should next be
continued along the algebraic Riemann surface determined by the germ
$Z_{\rm cell}$.  A natural tool would be a two-variable cell determinant
$Z_{\rm cell}(u,s)$, in which $s$ records integrated height; its singular set
may locate the first singularities of the finite part and relate them to
pressure at infinity.  The elementary height changes treated in
\cref{thm:height-covariance} are now controlled; a complementary problem is to
compare Abel and sharp finite parts under nonlinear changes of height or of a
stationary reference tail.

At the group level, edge indices alone do not determine the stabilizer and
centralizer corrections needed for an Euler product over primitive hyperbolic
conjugacy classes.  Establishing such a comparison, possibly with explicit
double-coset factors, is a separate problem; recent local-to-global determinant
formulas for tree actions offer a possible starting point
\cite{Marchionna2026}.  Extending the renormalized limit beyond stationary
cusp families will likewise require counterterms adapted to the growth and
accumulation of ends
\cite{BroiseParkkonenPaulin2019,BroiseParkkonenPaulin2021}.

Thus the failed global zeta leaves two computable invariants: an algebraic
local determinant for finite-core recurrence and a height-covariant
renormalized contribution from infinity.

\subsection*{Declaration of generative AI use}

During the preparation of this manuscript, the author used OpenAI ChatGPT and
Codex to assist with readability, Tikz drafting, and preliminary consistency checks.  The author independently
reviewed and verified the mathematical content and citations, made all final
revisions, and takes full responsibility for the manuscript.

\providecommand{\bysame}{\leavevmode\hbox to3em{\hrulefill}\thinspace}
\providecommand{\MR}{\relax\ifhmode\unskip\space\fi MR }
\providecommand{\MRhref}[2]{%
  \href{http://www.ams.org/mathscinet-getitem?mr=#1}{#2}
}
\providecommand{\href}[2]{#2}

\end{document}